\documentclass[a4paper]{article}

\usepackage[all]{xy}\usepackage[latin1]{inputenc}        
\usepackage[dvips]{graphics,graphicx}
\usepackage{amsfonts,amssymb,amsmath,color,mathrsfs, amstext}
\usepackage{amsbsy, amsopn, amscd, amsxtra, amsthm,authblk}
\usepackage{upref}
\usepackage[colorlinks,
linkcolor=red,
anchorcolor=red,
citecolor=red
]{hyperref}

\usepackage{geometry}
\usepackage{float}

\usepackage{yhmath}

\DeclareMathOperator{\su}{supp}

\def\e{\epsilon}

\numberwithin{equation}{section}              
\newtheorem{theorem}{Theorem}[section]
\newtheorem{lemma}{Lemma}[section]
\newtheorem{proposition}{Proposition}[section]
\newtheorem*{proposition*}{Proposition}
\newtheorem{corollary}{Corollary}[section]
\newtheorem*{corollary*}{Corollary}
\newtheorem{definition}{Definition}[section]
\newtheorem*{definitions*}{Definitions}
\newtheorem*{conjecture*}{\bf Conjecture}

\newtheorem*{example*}{\bf Example}
\theoremstyle{remark}
\newtheorem{remark}{\bf Remark}[section]

\begin{document}
\date{}                                     
\title{The modified Camassa-Holm equation in Lagrangian coordinates}
	
\author[1,2]{Yu Gao\thanks{yugao@hit.edu.cn}}
\author[2]{Jian-Guo Liu\thanks{jliu@phy.duke.edu}}
\affil[1]{Department of Mathematics, Harbin Institute of Technology, Harbin, 150001, P.R. China.}
\affil[2]{Department of Mathematics and Department of Physics, Duke University, Durham, NC 27708, USA.}

	\maketitle

\begin{abstract}
In this paper, we  study the modified  Camassa-Holm (mCH) equation in Lagrangian coordinates.  For some initial data $m_0$, we show that classical solutions to this equation blow up in finite time $T_{max}$.  Before $T_{max}$,  existence and uniqueness of classical solutions are established. Lifespan for classical solutions is obtained: $T_{max}\geq \frac{1}{||m_0||_{L^\infty}||m_0||_{L^1}}.$  And there is a unique solution $X(\xi,t)$ to the Lagrange dynamics which is a strictly monotonic function of $\xi$ for any $t\in[0,T_{max})$:  $X_\xi(\cdot,t)>0$.  As $t$ approaching  $T_{max}$, we prove that classical solution $m(\cdot ,t)$ in Eulerian coordinate has a unique limit $m(\cdot,T_{max})$  in Radon measure space and there is a point $\xi_0$ such that $X_\xi(\xi_0,T_{max})=0$ which means $T_{max}$ is an onset time of collision of  characteristics.  We also show that in some cases peakons are formed at $T_{max}$.  After $T_{max}$, we regularize the Lagrange dynamics to prove global existence of weak solutions $m$ in Radon measure space.
\end{abstract}

\section{Introduction}
In this work, we consider the following nonlinear partial differential equation in $\mathbb{R}$:
\begin{align}\label{mCH}
m_t+[(u^2-u^2_x)m]_x=0,\quad m=u-u_{xx},\quad x\in\mathbb{R},~~t>0,
\end{align}
 subject to an initial condition
\begin{align}\label{initial m}
m(x,0)=m_0(x).
\end{align}
This equation is referred to as the modified Camassa-Holm(mCH) eqaution with cubic nonlinearity, which was introduced as a new integrable system by several different researchers \cite{Fokas,Fuchssteiner,Olever,Qiao2}. It has a bi-Hamiltonian structure \cite{GuiLiuOlverQu,Olever} and a Lax-pair \cite{Qiao2}. 
Equation \eqref{mCH} also has solitary wave solutions of the form \cite{GuiLiuOlverQu}:
$$u(x,t)=pG(x-x(t)),\quad m(x,t)=p\delta(x-x(t)),~\textrm{ and }~x(t)=\frac{1}{6}p^2t,$$ 
where  $p$ is a constant representing the amplitude of the soliton and $G(x)=\frac{1}{2}e^{-|x|}$  is the  fundamental solution  for the Helmholtz operator $1-\partial_{xx}$. 
With this fundamental solution $G$, we have the following relation between functions $u$ and $m$:
\begin{align*}
u(x,t)=G\ast m=\int_{\mathbb{R}}G(x-y)m(y,t)dy.
\end{align*}
Moreover, global existence of $N$-peakon weak solutions of the following form was obtained in \cite{GaoLiu}:
$$u^N(x,t)=\sum_{i=1}^Np_iG(x-x_i(t)),~~m^N(x,t)=\sum_{i=1}^Np_i\delta(x-x_i(t)).$$

In the present paper, we study local well-posedness for classical solutions and global weak solutions to \eqref{mCH} in Lagrangian coordinates.  Below we introduce the Lagrange dynamics for the mCH equation. To this end, we first review the Lagrange dynamics for incompressible $2D$ Euler equation:
\begin{align*}
\left\{
\begin{array}{ll}
\omega_t(x,t)+\nabla\cdot\big(u(x,t)\omega(x,t)\big)=0,\quad (x,t)\in\mathbb{R}^2\times[0,\infty), \\
\omega(x,0)=\omega_0(x),
\end{array}
\right.
\end{align*}
where the velocity $u$ is determined from the vorticity $\omega$ by the Biot-Savart law
$$u(x,t)=\int_{\mathbb{R}^2}K_2(x-y)\omega(y,t)dy,\quad x\in\mathbb{R}^2,$$
involving the kernel $K_2(x)=(2\pi |x|^2)^{-1}(-x_2,x_1).$
Assume $X(\xi,t)$ is the flow map generated by the velocity field $u(x,t)$:
\begin{align*}
\left\{
\begin{array}{ll}
\dot{X}(\xi,t)=u(X(\xi,t),t), \quad \xi\in\mathbb{R}^2, ~~t>0,\\
X(\xi,0)=\xi.
\end{array}
\right.
\end{align*}
By the incompressible property $\nabla\cdot u=0$, we know
\begin{align}\label{eq:incompressEuler}
\omega(X(\xi,t),t)=\omega_0(\xi).
\end{align}
The $2D$ Euler equation can be rewritten in the Lagrange dynamics
\begin{align*}
\left\{
\begin{array}{ll}
\dot{X}(\xi,t)=u(X(\xi,t),t),~~X(\xi,0)=\xi\in\mathbb{R}^2,~~t>0,\\
\omega(X(\xi,t),t)=\omega_0(\xi),\\
u(x,t)=(K_2\ast \omega)(x,t).
\end{array}
\right.
\end{align*}

Comparing with the incompressible $2D$ Euler equation, assume $X(\xi,t)$ is the flow map for the mCH equation generated by the velocity field $u^2-u_x^2$:
$$\dot{X}(\xi,t)=(u^2-u_x^2)(X(\xi,t),t),~~X(\xi,0)=\xi\in\mathbb{R},~t>0.$$
In contrast with \eqref{eq:incompressEuler}, we have the following property for the mCH equation:
\begin{align*}
m(X(\xi,t),t)X_\xi(\xi,t)=m_0(\xi).
\end{align*}
Combining the above two equalities, the mCH equation \eqref{mCH} can be rewritten in the Lagrange dynamics:
\begin{align}\label{eq:lagrange1}
\begin{cases}
\dot{X}(\xi,t)=(u^2-u_x^2)(X(\xi,t),t),~~X(\xi,0)=\xi\in\mathbb{R},~t>0,\\
m(X(\xi,t),t)X_\xi(\xi,t)=m_0(\xi),\\
u(x,t)=(G\ast m)(x,t).
\end{cases}
\end{align}
Changing of variable gives
\begin{align*}
u(x,t)&=\int_{\mathbb{R}}G(x-y)m(y,t)dy=\int_{\mathbb{R}}G(x-X(\theta,t))m(X(\theta,t))X_\theta(\theta,t)d\theta\\
&=\int_{\mathbb{R}}G(x-X(\theta,t))m_0(\theta)d\theta.
\end{align*}
Set
\begin{align}\label{UXT}
U(x,t):&=u^2(x,t)-u_x^2(x,t)\nonumber\\
&=\left(\int_{\mathbb{R}}G(x-X(\theta,t))m_0(\theta)d\theta\right)^2-\left(\int_{\mathbb{R}}G_x(x-X(\theta,t))m_0(\theta)d\theta\right)^2.
\end{align}
Then, Equation \eqref{eq:lagrange1} can be rewritten as
\begin{align}\label{Lagran dynamics}
\left\{
\begin{array}{ll}
\dot{X}(\xi,t)=U(X(\xi,t),t),\\
X(\xi,0)=\xi\in\mathbb{R}.
\end{array}
\right.
\end{align}
When $m_0\in L^1(\mathbb{R})$, the following useful properties can be easily obtained:
\begin{equation}\label{uuxUbound}
|u(x,t)|\leq\frac{1}{2}||m_0||_{L^1},\quad |u_x(x,t)|\leq\frac{1}{2}||m_0||_{L^1}~\textrm{ and }~|U(x,t)|\leq\frac{1}{2}||m_0||_{L^1}^2.
\end{equation}

In the rest of this paper, we assume the initial $m_0$ satisfying $\su\{m_0\}\subset(-L,L)$ for some constant $L>0$. Next, we summarize our  main results in four theorems. 
\begin{theorem}\label{maintheorem1}
Suppose $m_0\in C_c^k(-L,L)$ $(k\in \mathbb{N}, k\geq1)$. Then, there exists a unique maximum existence time $T_{max}\leq+\infty$ such that
 Lagrange dynamics \eqref{Lagran dynamics} has a unique solution
$$X\in C_1^{k+1}([-L,L]\times[0,T_{max})),$$
which satisfies
\begin{align*}
X_\xi(\xi,t)>0~\textrm{ for }~(\xi,t)\in[-L,L]\times[0,T_{max}).
\end{align*}
(The solution space is defined by \eqref{functionspace}.) The mCH equation \eqref{mCH}-\eqref{initial m} has a unique classical solution
$$u\in C^{k+2}_1(\mathbb{R}\times[0,T_{max})),\quad m\in C^{k}_1(\mathbb{R}\times[0,T_{max})),$$
which can be represented by $X(\xi,t)$ as
\begin{align}\label{eq:umdef}
u(x,t)=\int_{-L}^LG(x-X(\theta,t))m_0(\theta)d\theta~\textrm{ and }~m(x,t)=\int_{-L}^L\delta(x-X(\theta,t))m_0(\theta)d\theta.
\end{align}
Moreover, $m$ satisfies:
\begin{align}\label{eq:support}
\su\{m(\cdot ,t)\}\subset(-L,L)~\textrm{ for }~t\in[0,T_{max}).
\end{align}

If $T_{max}<+\infty$, then the following holds:\\
$\mathrm{(i)}$ We have
$$X\in C_1^{k+1}([-L,L]\times[0,T_{max}]).$$
$\mathrm{(ii)}$
The following equivalent statements hold:

\emph{(a)}
$$\limsup_{t\rightarrow T_{max}}||m(\cdot ,t)||_{L^\infty}=+\infty,$$

\emph{(b)}
\begin{gather*}
\left\{\begin{split}
&X_\xi(\xi,t)>0~\textrm{ for }~(\xi,t)\in[-L,L]\times[0,T_{max});\\
&\min_{\xi\in[-L,L]}X_\xi(\xi,T_{max})=0.
\end{split}
\right.
\end{gather*}

\emph{(c)}
\begin{gather*}
\liminf_{t\rightarrow T_{max}}\Big\{\inf_{\xi\in[-L,L]}\int_0^t(mu_x)(X(\xi,s),s)ds\Big\}=-\infty,
\end{gather*}

\emph{(d)}
\begin{equation*}
\liminf_{t\rightarrow T_{max}}\Big\{\inf_{x\in\mathbb{R}}(mu_x)(x,t)\Big\}=-\infty,
\end{equation*}

\emph{(e)}
\begin{equation*}
\limsup_{t\rightarrow T_{max}}||m(\cdot ,t)||_{W^{1,p}}=+\infty,~\textrm{ for }p\geq1,
\end{equation*}

\emph{(f)}
\begin{align*}
\int_0^{T_{max}}||m(\cdot ,t)||_{L^\infty}dt=+\infty.
\end{align*}
$\mathrm{(iii)}$
There exists a unique function $u(\cdot ,T_{max})$ such that
\begin{align*}
\lim_{t\rightarrow T_{max}}u(x,t)=u(x,T_{max}),\quad \lim_{t\rightarrow T_{max}}u_x(x,t)=u_x(x,T_{max})\textrm{ for every }x\in\mathbb{R}.
\end{align*}
Moreover,  for any $t\in[0,T_{max}]$ we have
$$u(\cdot,t),~~u_x(\cdot,t)\in BV(\mathbb{R})$$
and
\begin{align*}
\mathrm{Tot.Var.}\{u(\cdot, t)\}\leq M_1,\quad \mathrm{Tot.Var.}\{u_x(\cdot, t)\}\leq 2M_1.
\end{align*}
Here, $BV(\mathbb{R})$ is the space of functions with bounded variation (see definition \ref{BV}).\\
$\mathrm{(iv)}$
There exists a unique $m(\cdot ,T_{max})\in\mathcal{M}(\mathbb{R})$ (Radon measure space on $\mathbb{R}$) such that
$$m(\cdot ,t)\stackrel{\ast}{\rightharpoonup} m(\cdot , T_{max})~\textrm{ in }~\mathcal{M}(\mathbb{R}),~\textrm{ as }~t\rightarrow T_{max}.$$

\end{theorem}

(a) and (b) tells us that $T_{max}$ is an onset time of collisions of characteristics. \eqref{eq:support} implies that the supports for classical solutions will not change.

Our another main theorem is about finite time blow-up behaviors and the lifespan of classical solutions. Let $T_{max}(m_0)$ be the maximum existence time of the classical solution to the mCH equation subject to an initial condition $m_0$. Then we have the following theorem about lifespan for classical solutions.
\begin{theorem}\label{lifespan}
Assume $m_0\in C_c^k(\mathbb{R})$ $(k\in \mathbb{N}, k\geq1)$.\\
$\mathrm{(i)}$
We have
\begin{align}\label{lifespanlowbound}
T_{max}(m_0)\geq\frac{1}{||m_0||_{L^\infty}||m_0||_{L^1}}.
\end{align}
$\mathrm{(ii)}$
If there exists  $\overline{\xi}\in[-L,L]$ such that
\begin{align}\label{initialu0}
m_0(\overline{\xi})\partial_xu_0(\overline{\xi})<0,\quad |m_0(\overline{\xi})|(\partial_xu_0(\overline{\xi}))^2>\frac{1}{2}||m_0||^3_{L^1},
\end{align}
then the classical solution to the mCH equation will blow up in finite time.
Moreover, for any $\epsilon>0$ we have
\begin{align}\label{lifespanupbound}
\frac{1}{||m_0||_{L^\infty}||m_0||_{L^1}}\cdot \frac{1}{\epsilon^2}\leq T_{max}(\epsilon m_0)\leq \frac{1}{||m_0||^2_{L^1}}\cdot\frac{1}{\epsilon^2}.
\end{align}

\end{theorem}

This theorem implies that there are smooth initial data with arbitrary small support and  arbitrary small $C^k(\mathbb{R})$-norm, $k\in \mathbb{N}$, for which the classical solution does not exist globally.

Next, we give a theorem to show the formation of peakons at finite blow-up time $T_{max}$. From Theorem \ref{maintheorem1}, we know there is a point $\xi_0\in[-L,L]$ such that $X_\xi(\xi_0,T_{max})=0$. Set
$$F_{T_{max}}:=\{X(\xi,T_{max}):\xi\in [-L,L],~~X_\xi(\xi,T_{max})=0\}.$$
 For any $x\in F_{T_{max}}$, because $X_\xi(\cdot,T_{max})\ge0$, we know that $X^{-1}(x,T_{max})$ is either a single point or a closed interval.
Denote
$$\widehat{F}_{T_{max}}:=\{x\in F_{T_{max}}:X^{-1}(x,T_{max})=[\xi_1,\xi_2] ~\textrm{ for some }~\xi_1<\xi_2\}.$$
The figure below describe these singular points.
\begin{figure}[H]
\begin{center}
\includegraphics[width=0.5\textwidth]{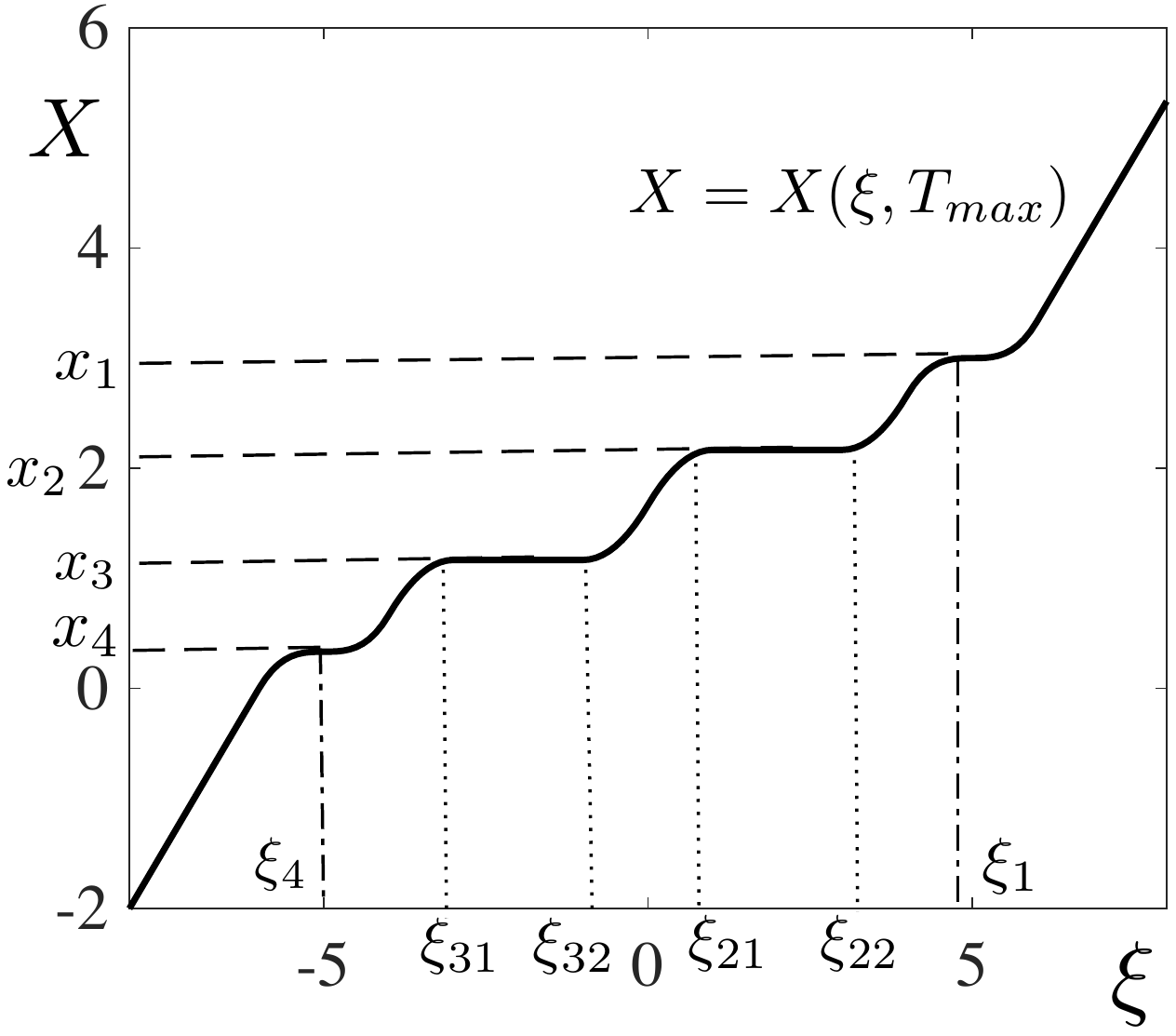}
\end{center}
\caption{At $T_{max}$, $X_\xi(\cdot,T_{max})\geq0$ and $X_\xi(\xi,T_{max})=0$ for $\xi\in\{\xi_1,\xi_4\}\cup[\xi_{21},\xi_{22}]\cup[\xi_{31},\xi_{32}]$. $F_{T_{max}}=\{x_1,x_2,x_3,x_4\}$ and $\widehat{F}_{T_{max}}=\{x_2,x_3\}$.}
\label{fig:step}
\end{figure}
For $x\in \widehat{F}_{T_{max}}$ and $X^{-1}(x,T_{max})=[\xi_1,\xi_2]$, we show that $m_0$ will not change sign in $[\xi_1,\xi_2]$ (see Proposition \ref{discontinuousset}). 
Hence, $\int_{\xi_1}^{\xi_2}m_0(\xi)d\xi\neq 0$.
We have the following theorem.
\begin{theorem}\label{maintheorem3}
Assume $F_{T_{max}}=\{x_i\}_{i=1}^{N_1}$  and  $\widehat{F}_{T_{max}}=\{x_i\}_{i=1}^N$ ($N\leq N_1$). Let $X^{-1}(x_i,T_{max})=[\xi_{i1},\xi_{i2}]$ and $p_i=\int_{\xi_{i1}}^{\xi_{i2}}m_0(\xi)d\xi$ for $1\leq i\leq N$. Then
\begin{align*}
m(x,T_{max})=m_1(x)+\sum_{i=1}^Np_i\delta(x-x_i)
\end{align*}
where  $m_1\in L^1(\mathbb{R})$ is given by \eqref{definitionofm1}.
\end{theorem}

At last, we give a theorem to show global existence of weak solutions (see Definition \ref{weaksolution}). Theorem \ref{maintheorem1} (iv) tells that classical solutions become  Radon measures when blow-up happens. After the blow-up time $T_{max}$, we can extend our solution $m(x,t)$ globally in the Radon measure space. We have: 
\begin{theorem}\label{maintheorem4}
Let $m_0\in\mathcal{M}(\mathbb{R})$ with compact support. Then there exists a global weak solution to the mCH equation satisfying:
$$u\in C([0,+\infty);H^1(\mathbb{R}))\cap L^\infty(0,+\infty; W^{1,\infty}(\mathbb{R})),$$
and
$$m=u-u_{xx}\in\mathcal{M}(\mathbb{R}\times[0,T))~\textrm{ for any }~T>0.$$
\end{theorem}

Now, we compare the mCH equation with the Camassa-Holm (CH) equation:
$$\partial_tm+\partial_x(um)+m\partial_xu=0,\quad m=u-u_{xx},\quad x\in\mathbb{R},t>0.$$
The CH equation was established by Camassa and Holm \cite{Camassa} to model the unidirectional propagation of waves at free surface of a shallow layer of water ($u(x,t)$ representing the height of water's free surface above a flat bottom). It is also a complete integrable system which has a bi-Hamiltonian structure and a Lax pair \cite{Camassa}.

There are some different properties between the CH equation and the mCH equation.

$\bullet$ \emph{Classical solutions and blow-up criteria.} 
For a large class of initial data,  classical solutions to the CH equation blow up in  finite time (see \cite{Brandolese} and references in it).  Moreover, the only way that a classical solution of the CH equation fails to exist globally is that the wave breaks \cite{Constantin11} in the sense that the solution $u$ remains bounded  while the spatial derivative $u_x$ becomes unbounded. For the mCH equation, blow-up behaviors also happen for a large class of initial data (see \cite{Chenrongming,GuiLiuOlverQu,LiuOlver}). However, $u_{xx}$ (hence $m$) becomes unbounded when blow-up happens, while $u$ and  $u_x$ remain bounded.  

$\bullet$ \emph{Lifespan for classical solutions.} 
Comparing with \eqref{lifespanupbound}, the lower bound for lifespan of strong solutions to the CH equation with initial data $\e u_0(x)$ is given by \cite{Danchin,Luc}:
$$T_{max}\geq \frac{C}{\e}~\textrm{ for some  }~C=C(u_0).$$

$\bullet$ \emph{$N$-peakon weak solutions.} 
Trajectories for $N$-peakon weak solutions to the CH equation never collide \cite{CamassaLee,AlinaLiu} provided that the initial datum $m_0(x)=\sum_{i=1}^Np_i\delta(x-c_i)$ satisfies $p_i>0$ and $c_i\neq c_j$ for $i\neq j$. However, the trajectories for $N$-peakon solutions of the mCH equation may collide in finite time even if $m_0\geq0$ \cite{GaoLiu}. Moreover, for the CH equation, when blow-up happens at finite time $T_{max}$, we have ${\liminf}_{t\to T_{max}}u_x(x,t)=-\infty$ (see \cite{Constantin11,Luc}). Peakon solutions $u$ and its derivative $u_x$ are in BV space, which are bounded functions. Hence, peakon solutions can not be formed when blow-up happens (comparing with Theorem \ref{maintheorem3}) for the CH equation.

$\bullet$ \emph{General weak solutions.} 
In \cite{GaoLiu}, the authors proved nonuniqueness of weak solutions obtained by Theorem \ref{maintheorem4}. Comparing with Theorem \ref{maintheorem4}, there is a unique global weak solution $u\in C([0,+\infty);H^1(\mathbb{R}))$ and $m\in\mathcal{M}_+(\mathbb{R})$ (see \cite{AlinaLiu,Constantin}) to the CH equation when $u_0\in H^1(\mathbb{R})$ and $0\leq m_0\in\mathcal{M}(\mathbb{R})$.  
For general initial data $u_0\in H^1(\mathbb{R})$, global existence of weak solutions to the CH equation was obtained by several different methods (see \cite{Bressan2,Bressan3,Holden,Holden2,Xin,Xin2}).

For more results about local well-posedness and blow up behavior of strong solutions to the Cauchy problem \eqref{mCH}-\eqref{initial m}, one can refer to \cite{Chenrongming,FuGui,GuiLiuOlverQu,LiuOlver}. For weak solutions, one can refer to \cite{GaoLiu,Qingtianzhang}.

The rest of this article is organized as follows.
In Section \ref{sec1}, we use contraction mapping theorem to prove  local existence and uniqueness of solutions $X(\xi,t)$ to the Lagrange dynamics \eqref{Lagran dynamics}. Then, we use $X(\xi,t)$ to give $(u(x,t),m(x,t))$ and prove that it is a unique classical solution to the mCH equation \eqref{mCH}-\eqref{initial m}. Besides, when $\sup_{t\in[0,T)}||m(\cdot ,t)||_{L^\infty}$ is finite, we can extend this classical solution in time. In Section \ref{sec2}, we show some blow-up criteria for classical solutions. In Section \ref{sec3}, we  prove that for some initial data classical solutions blow up in a finite time and the estimates for blow-up rates are given. For small initial data, almost global existence of classical solutions is obtained. In Section \ref{sec4}, we study  classical solutions at blow-up time $T_{max}$. $u(\cdot ,T_{max})$ and $u_x(\cdot ,T_{max})$ are BV functions while $m(\cdot ,t)$ has a unique limit $m(\cdot, T_{max})$ in Radon measure space as $t\rightarrow T_{max}$. Moreover, we prove that in some cases peakons are formed at $T_{max}$. In the last section, we use regularized Lagrange dynamics to prove global existence of weak solutions in Radon measure space.

\section{Lagrange dynamics and short time classical solutions}\label{sec1}
In this section, we study the existence and uniqueness of solutions to Lagrange dynamics \eqref{Lagran dynamics}. Then, we prove $(u(x,t),m(x,t))$ defined by \eqref{eq:umdef} is a unique classical solution to \eqref{mCH}-\eqref{initial m}.

First, let's introduce the spaces for solutions. For nonnegative integers $k,n$ and real number $T>0$, we denote $$U_T:=[-L,L]\times[0,T]$$
 and the function space
\begin{align}\label{functionspace}
C^{k}_{n}(U_T):=\big\{u:U_T\rightarrow\mathbb{R}:\quad\partial_x^\beta u\in C(U_T),\quad|\beta|\leq k;\quad\partial_t^\alpha u\in C(U_T),\quad|\alpha|\leq n\big\}.
\end{align}
Similarly, we can define $C_n^k(\mathbb{R}\times[0,T])$.

We will present the results of this section in three subsections  as follows.
\begin{enumerate}
  \item In Subsection \ref{section1}, when $m_0\in C_c^k(-L,L)$, we prove local existence and uniqueness of a solution $X\in C^{k+1}_1([-L,L]\times[0,t_1])$ to \eqref{Lagran dynamics} such that $$\min\{X_\xi(\xi,t):(\xi,t)\in[-L,L]\times[0,t_1]\}>0.$$
  \item In Subsection \ref{section2}, we prove $u$ defined by \eqref{eq:umdef} belongs to $C_1^{k+2}(\mathbb{R}\times[0,t_1])$ and ($u,m$) is a unique classical solution to the mCH equation.
  \item In Subsection \ref{section3}, we prove that whenever the classical solution $m$ satisfies
  $$\sup_{t\in[0,T)}||m(\cdot,t)||_{L^\infty}<\infty,$$ 
  we can extend the classical solution in time.
\end{enumerate}

\subsection{Local existence and uniqueness of solutions to Lagrange dynamics}\label{section1}
In this subsection, we use the contraction mapping theorem to prove short time existence and uniqueness of solutions to the Lagrange dynamics \eqref{Lagran dynamics}, which is equivalent to the following integral equation:
\begin{align}\label{equi form}
X(\xi,t)=\xi+\int_0^tU(X(\xi,s),s)ds,
\end{align}
where $U$ is defined by \eqref{UXT}. Set
\begin{align}\label{contractionmapping}
T_X(\xi,t):=\xi+\int_0^tU(X(\xi,s),s)ds.
\end{align}
For  constants $C_2>C_1>0$ and $t_1>0$, we define
\begin{align}\label{solutionspace}
\mathcal{Q}_{t_1}(C_1,C_2):=\Big\{X\in C(U_{t_1})&:C_1(\xi-\eta)^2\leq (X(\xi,t)-X(\eta,t))(\xi-\eta)\nonumber\\
&\leq C_2(\xi-\eta)^2,\textrm{for any }\xi,\eta\in[-L,L] \textrm{ and } t\in[0,t_1]\Big\}.
\end{align}
Obviously, $\mathcal{Q}_{t_1}(C_1,C_2)$ is a closed subset of $ C(U_{t_1})$. We will look for suitable constants $C_1,~C_2,~ t_1$ and then use the contraction mapping theorem in the set $\mathcal{Q}_{t_1}(C_1,C_2)$.

Before presenting the existence and uniqueness theorem, we give two useful lemmas.

\begin{lemma}\label{continuous lemma0}
Assume $g\in L^\infty(-L,L)$ and $X(\xi,t)\in \mathcal{Q}_{t_1}(C_1,C_2)$ for some constants  $C_2>C_1>0$ and $t_1>0$.
Let
$$A(x,t):=\int_{-L}^LG'(x-X(\theta,t))g(\theta)d\theta.$$
Then, we have $A\in C(\mathbb{R}\times[0,t_1]).$
\end{lemma}
\begin{proof}
According to \eqref{solutionspace}, $X(\xi,t)$ is monotonic about $\xi$. For given $(x,t)\in U_{t_1}$, we separate the proof into three parts.

\emph{Step 1}.  Continuity at $(x,t)\in\mathbb{R}\times[0,t_1]$ when $x>X(L,t)$.

For $(y,s)$ closed to $(x,t)$ and because $X\in C(U_{t_1})$ is monotonic, we can assume $y>X(\theta,s)$ for $\theta\in(-L,L)$. A direct estimate gives
\begin{align*}
&\quad|A(y,s)-A(x,t)|=\bigg|\int_{-L}^LG'(y-X(\theta,s))g(\theta)d\theta-\int_{-L}^LG'(x-X(\theta,t))g(\theta)d\theta\bigg|\\
&\leq\int_{-L}^L|G(x-X(\theta,t))-G(y-X(\theta,s))|\cdot|g(\theta,t)|d\theta\\
&\leq \frac{1}{2}||g||_{L^\infty}\int_{-L}^L|y-x|+|X(\theta,t)-X(\theta,s)|d\theta.
\end{align*}
Therefore, according to the uniform continuity of $X$, $A$ is continuous at $(x,t)$. The proof of the case $x<X(-L,t)$ is similar.

\emph{Step 2}. Continuity at $(x,t)\in\mathbb{R}\times[0,t_1]$ when $x=X(\xi,t)$ for some $\xi\in(-L,L)$.

Due to the continuity of $X$, for $(y,s)$ closed to $(x,t)$, there exists $\eta\in[-L,L]$ such that $X(\eta,s)=y$. Without lose of generality, we assume $\xi>\eta.$
\begin{align*}
&\quad|A(y,s)-A(x,t)|\\
&=\bigg|\int_{-L}^LG'(X(\eta,s)-X(\theta,s))g(\theta)d\theta-\int_{-L}^LG'(X(\xi,t)-X(\theta,t))g(\theta)d\theta\bigg|\\
&\leq\int_{-L}^\eta|G'(X(\eta,s)-X(\theta,s))-G'(X(\xi,t)-X(\theta,t))||g(\theta)|d\theta\\
&\quad+\int_{\eta}^\xi|G'(X(\eta,s)-X(\theta,s))-G'(X(\xi,t)-X(\theta,t))||g(\theta)|d\theta\\
&\quad+\int_{\xi}^L|G'(X(\eta,s)-X(\theta,s))-G'(X(\xi,t)-X(\theta,t))||g(\theta)|d\theta.
\end{align*}
Then, the monotonicity of $X(\theta,t)$ implies that
\begin{align*}
|A(y,s)-A(x,t)|\leq||g||_{L^\infty}|\xi-\eta|+||g||_{L^\infty}\int_{-L}^L|x-y|+|X(\theta,s)-X(\theta,t)|d\theta.
\end{align*}
From the definition of $\mathcal{Q}_{t_1}(C_1,C_2)$, we have
\begin{align}\label{will be used}
|x-y|=|X(\xi,t)-X(\eta,s)|&\geq|X(\xi,s)-X(\eta,s)|-|X(\xi,t)-X(\xi,s)|\nonumber\\
&\geq C_1|\xi-\eta|-|X(\xi,t)-X(\xi,s)|.
\end{align}
Therefore, $|\xi-\eta|\leq \frac{1}{C_1}(|x-y|+|X(\xi,t)-X(\xi,s)|)$.
Hence,  $A(x,t)$ is continuous at $(x,t)$.

\emph{Step 3}. Continuity at $(x,t)\in\mathbb{R}\times[0,t_1]$ when $x=X(L,t)$. The case $x=X(-L,t)$ is similar.

For $(y,s)$ closed to $(x,t)$, we have two cases. When $y>X(L,s)$, we can use Step 1. When there exists $\xi\in(-L,L)$ such that $y=X(\xi,s)$, we can use Step 2.

This is the end of the proof.
\end{proof}

\begin{lemma}\label{contractionmap1}
Assume $m_0\in L^\infty(-L,L)$ and $X\in\mathcal{Q}_{t_1}(C_1,C_2)$ for some constants $C_2>C_1>0$ and $t_1>0$. Then, for $-L\leq\eta<\xi\leq L$, we have
\begin{align}\label{incre}
[1-(M_1M_\infty+C_2M_1^2)t_1](\xi-\eta)&\leq T_X(\xi,t)-T_X(\eta,t)\nonumber\\
&\leq[1+(M_1M_\infty+C_2M_1^2)t_1](\xi-\eta),
\end{align}
where $M_1:=||m_0||_{L^1}$ and $M_\infty:=||m_0||_{L^\infty}$.

\end{lemma}

\begin{proof}
Assume $X\in\mathcal{Q}_{t_1}(C_1,C_2)$ for some constants $C_2>C_1>0$ and $t_1>0$. For $-L\leq\eta<\xi\leq L$, $t\in[0,t_1],$ we have
\begin{align}\label{incre1}
T_X(\xi,t)-T_X(\eta,t)=\xi-\eta+\int_0^t[U(X(\xi,s),s)-U(X(\eta,s),s)]ds.
\end{align}
By \eqref{uuxUbound}, we obtain
\begin{align*}
&\quad|U(X(\xi,s),s)-U(X(\eta,s),s)|\\
&\leq |u^2(X(\xi,s),s)-u^2(X(\eta,s),s)|+|u_x^2(X(\xi,s),s)-u_x^2(X(\eta,s),s)|\\
&\leq M_1|u(X(\xi,s),s)-u(X(\eta,s),s)|+M_1|u_x(X(\xi,s),s)-u_x(X(\eta,s),s)|\\
&=:I_1+I_2.
\end{align*}
Because $X\in\mathcal{Q}_{t_1}(C_1,C_2)$, we have
\begin{align*}
|u(X(\xi,s),s)-u(X(\eta,s),s)|&=\bigg|\int_{-L}^Lm_0(\theta)\Big(G(X(\xi,s)-X(\theta,s))-G(X(\eta,s)-X(\theta,s))\Big)d\theta\bigg|\\
&\leq\frac{1}{2}M_1(X(\xi,s)-X(\eta,s))\leq\frac{1}{2}M_1C_2(\xi-\eta).
\end{align*}
Thus,
$$I_1\leq\frac{1}{2}M_1^2C_2(\xi-\eta).$$

Next, we estimate $I_2$.

 When $X\in\mathcal{Q}_{t_1}(C_1,C_2)$, we have
$(X(\xi,s)-X(\theta,s))(X(\eta,s)-X(\theta,s))>0$ for $\theta\in[-L,\eta)\cap(\xi,L]$. On the other hand, we know $|G'(a)-G'(b)|=|G(a)-G(b)|\leq \frac{1}{2}|a-b|$ when $ab>0$. Therefore, 
\begin{align*}
&\quad|u_x(X(\xi,s),s)-u_x(X(\eta,s),s)|\\
&\leq\int_{[-L,\eta)\cap(\xi,L]}m_0(\theta)|G'(X(\xi,s)-X(\theta,s))-G'(X(\eta,s)-X(\theta,s))|d\theta\\
&\quad+\int_\eta^\xi m_0(\theta)|G'(X(\xi,s)-X(\theta,s))-G'(X(\eta,s)-X(\theta,s))|d\theta\\
&\leq(\frac{1}{2}M_1C_2+M_\infty)(\xi-\eta).
\end{align*}
Thus
$$I_2\leq (\frac{1}{2}C_2M_1^2+M_1M_\infty)(\xi-\eta).$$
Combining $I_1$ and $I_2$ gives
\begin{align*}
-(C_2M_1^2+M_1M_\infty)t_1(\xi-\eta)&\leq\int_0^t[U(X(\xi,s),s)-U(X(\eta,s),s)]ds\\
&\leq (C_2M_1^2+M_1M_\infty)t_1(\xi-\eta).
\end{align*}
Together with \eqref{incre1}, we obtain \eqref{incre}.
\end{proof}

We have the following existence and uniqueness theorem.
\begin{theorem}\label{noregularityresults}
Assume $m_0\in C_c^k(-L,L)$ $(k\in \mathbb{N}, k\geq1)$. Let $M_1:=||m_0||_{L^1}$ and $M_\infty:=||m_0||_{L^\infty}$. Then, for any $t_1$ with 
\begin{align}\label{time}
0<t_1<\frac{1}{2M_1^2+M_1M_\infty},
\end{align}
there exist constants $C_2>C_1>0$  satisfying
\begin{align}\label{C_2}
\frac{1+M_1M_\infty t_1}{1-M_1^2t_1}<C_2<\frac{1-M_1M_\infty t_1}{M_1^2t_1},
\end{align}
and
\begin{align}\label{C_1}
0<C_1<1-(M_1M_\infty+C_2M_1^2)t_1,
\end{align}
such that \eqref{equi form} has a unique solution $X\in C_1^{k+1}(U_{t_1})$ satisfying
\begin{align}
C_1\leq X_\xi(\xi,t)\leq C_2\label{Xxiproperties}
\end{align}
for $(\xi,t)\in[-L,L]\times[0,t_1]$.

Moreover, for any $\ell\in\mathbb{N}$, $0\leq\ell\leq k+1$, there exists a constant $\widehat{C}_\ell$ (depending on $||m_0||_{C^k}$, $||m_0||_{L^1}$ and $t_1$) such that
\begin{align}\label{derivativeproperties}
|\partial_\xi^\ell X(\xi,t)|\leq \widehat{C}_\ell.
\end{align}
\end{theorem}

\begin{proof}
We separate this proof into two parts.

\textbf{Part I.(Existence and Uniqueness)} We use the contraction mapping theorem to prove the existence of a unique solution $X\in C_1^0(U_{t_1})$ to \eqref{equi form}.

\emph{Step 1}. When $0<t_1<\frac{1}{2M_1^2+M_1M_\infty}$, we prove there are constants $C_2>C_1>0$ such that when $X\in\mathcal{Q}_{t_1}(C_1,C_2)$, we have $T_X\in \mathcal{Q}_{t_1}(C_1,C_2)$, where $T_X$ is defined by \eqref{contractionmapping}.

When $t_1$ satisfies \eqref{time}, we have
\begin{align}\label{contrac}
M_1^2t_1<\frac{1}{2}\textrm{ and }M_1M_\infty t_1<1.
\end{align}
A simple computation shows that
$$\frac{1+M_1M_\infty t_1}{1-M_1^2t_1}<\frac{1-M_1M_\infty t_1}{M_1^2t_1}.$$
Hence, there is a constant $C_2$ satisfying \eqref{C_2}.
Moreover, inequality \eqref{C_2} implies
\begin{align}\label{C22}
1+(M_1M_\infty+C_2M_1^2)t_1\leq C_2,
\end{align}
and
$$0<1-(M_1M_\infty+C_2M_1^2)t_1.$$
Therefore, we can choose $C_1$ satisfying \eqref{C_1}.

When $X\in\mathcal{Q}_{t_1}(C_1,C_2)$, combining \eqref{incre}, \eqref{C_1} and \eqref{C22} gives
$$T_X\in\mathcal{Q}_{t_1}(C_1,C_2)$$
and Step 1 is completed.

\emph{Step 2}. We prove $T_X$ is a contraction map on $\mathcal{Q}_{t_1}(C_1,C_2)$.

For $X,Y\in\mathcal{Q}_{t_1}(C_1,C_2)$, combining \eqref{uuxUbound} we have
\begin{align}\label{contraction1}
&\quad|T_X(\xi,t)-T_Y(\xi,t)|\leq\int_0^t|U(X(\xi,s),s)-U(Y(\xi,s),s)|ds\nonumber\\
&\leq M_1\int^t_0|u(X(\xi,s),s)-u(Y(\xi,s),s)|ds+M_1\int^t_0|u_x(X(\xi,s),s)-u_x(Y(\xi,s),s)|ds\nonumber\\
&=:J_1+J_2.
\end{align}
For the first term $J_1$, we estimate
\begin{align}\label{contraction2}
 u(X(\xi,s),s)-u(Y(\xi,s),s)&=\int_{-L}^Lm_0(\theta)\Big(G(X(\xi,s)-X(\theta,s))-G(Y(\xi,s)-Y(\theta,s))\Big)d\theta\nonumber\\
&\leq\frac{1}{2}\int_{-L}^Lm_0(\theta)(|X(\xi,s)-Y(\xi,s)|+|X(\theta,s)-Y(\theta,s)|)d\theta\nonumber\\
&\leq M_1||X-Y||_{C(U_{t_1})}.
\end{align}
For the second term, due to $(X(\xi,s)-X(\theta,s))(Y(\xi,s)-Y(\theta,s))>0$, we obtain
\begin{align}\label{contraction3}
u_x(X(\xi,s),s)-u_x(Y(\xi,s),s)&=\int_{-L}^Lm_0(\theta)\Big(G'(X(\xi,s)-X(\theta,s))-G'(Y(\xi,s)-Y(\theta,s))\Big)d\theta\nonumber\\
&\leq M_1||X-Y||_{C(U_{t_1})}.
\end{align}
Combining \eqref{contraction1}, \eqref{contraction2}, and \eqref{contraction3}, we have
\begin{align*}
|T_X(\xi,t)-T_Y(\xi,t)|\leq J_1+J_2\leq 2M_1^2t_1||X-Y||_{C(U_{t_1})},
\end{align*}
which implies
$$||T_X-T_Y||_{C(U_{t_1})}\leq 2M_1^2t_1||X-Y||_{C(U_{t_1})}.$$
Inequality \eqref{contrac} shows that $T_X$ is a contraction map.

At last, by the contraction mapping theorem, the system \eqref{equi form} (or \eqref{Lagran dynamics}) has a unique solution in $C(U_{t_1})$.

On the other hand, using Lemma \ref{continuous lemma0} we can see $U=u^2-u_x^2\in C(\mathbb{R}\times[0,t_1])$, which means
$$\partial_tX\in C(U_{t_1}).$$
Hence, $X\in C_1^0(U_{t_1})$ and Part I is finished.

\textbf{Part II. (Regularity)} We show  $X$ obtained in the first part belongs to $C_1^{k+1}(U_{t_1}).$

From the first part, we can see solution $X$  belongs to $C_1^0(U_{t_1}).$
For this solution we have the following properties
$$X(\xi,t)-X(\theta,t)>0, \quad-L<\theta<\xi;\quad X(\xi,t)-X(\theta,t)<0,\quad\xi<\theta<L.$$
On the other hand,  $G(x)=\frac{1}{2}e^{-|x|}$ satisfies:
$$G'(x)=G(x), \quad x<0;\quad G'(x)=-G(x), \quad x>0.$$
We obtain
\begin{align}
\int_{-L}^\xi G'(X(\xi,s)-X(\theta,s))m_0(\theta)d\theta=-\int_{-L}^\xi G(X(\xi,s)-X(\theta,s))m_0(\theta)d\theta\label{1},\\
\int_{\xi}^L G'(X(\xi,s)-X(\theta,s))m_0(\theta)d\theta=\int_\xi^L G(X(\xi,s)-X(\theta,s))m_0(\theta)d\theta,\label{2}
\end{align}
Hence,
\begin{align*}
u_x(X(\xi,t),t)&=\int_{-L}^\xi G'(X(\xi,t)-X(\theta,t))m_0(\theta)d\theta+\int_{\xi}^LG'(X(\xi,t)-X(\theta,t))m_0(\theta)d\theta\\
&=-\int_{-L}^\xi G(X(\xi,t)-X(\theta,t))m_0(\theta)d\theta+\int_{\xi}^LG(X(\xi,t)-X(\theta,t))m_0(\theta)d\theta.
\end{align*}
We obtain
\begin{align}\label{short representation}
&\quad U(X(\xi,t),t)=u^2(X(\xi,t),t)-u_x^2(X(\xi,t),t)\nonumber\\
&=4\bigg(\int_{-L}^\xi G(X(\xi,t)-X(\theta,t))m_0(\theta)d\theta\bigg)\bigg(\int_{\xi}^LG(X(\xi,t)-X(\theta,t))m_0(\theta)d\theta\bigg).
\end{align}
Thus
\begin{align}\label{uniqueXT}
X(\xi,t)=\xi+4\int_0^t \bigg(&\int_{-L}^\xi G(X(\xi,s)-X(\theta,s))m_0(\theta)d\theta\bigg)\bigg(\int_{\xi}^LG(X(\xi,s)-X(\theta,s))m_0(\theta)d\theta\bigg)ds.
\end{align}
Because $X(\xi,t)$ is monotonic about $\xi$, its derivative exists for a.e. $\xi\in[-L,L]$. Differentiating with respect to $\xi$ shows that for a.e. $\xi\in[-L,L]$,
\begin{multline}\label{Xxi}
X_\xi(\xi,t)=1+4G(0)m_0(\xi)\int_0^t\bigg(\int_\xi^L G(X(\xi,s)-X(\theta,s))m_0(\theta)d\theta\\
-\int_{-L}^\xi G(X(\xi,s)-X(\theta,s))m_0(\theta)d\theta\bigg)ds\\
+4\int_0^t X_\xi(\xi,s)\bigg(\int_{-L}^\xi G'(X(\xi,s)-X(\theta,s))m_0(\theta)d\theta\bigg)\bigg(\int_{\xi}^LG(X(\xi,s)-X(\theta,s))m_0(\theta)d\theta\bigg)ds\\
+4\int_0^t X_\xi(\xi,s)\bigg(\int_{-L}^\xi G(X(\xi,s)-X(\theta,s))m_0(\theta)d\theta\bigg)\bigg(\int_{\xi}^LG'(X(\xi,s)-X(\theta,s))m_0(\theta)d\theta\bigg)ds,
\end{multline}
Due to \eqref{1} and \eqref{2}, the sum of the last two terms in \eqref{Xxi}  is zero, which leads to
\begin{align}\label{Xxixit}
&\quad X_\xi(\xi,t)=1+2m_0(\xi)\int_0^t\bigg(\int_\xi^L G(X(\xi,s)-X(\theta,s))m_0(\theta)d\theta\nonumber\\
&\qquad\qquad -\int_{-L}^\xi G(X(\xi,s)-X(\theta,s))m_0(\theta)d\theta\bigg)ds.
\end{align}
Because $m_0\in C_c^k(-L,L)$, we have $X_\xi\in C(U_{t_1})$ which means $X\in C^1_1(U_{t_1})$.

From \eqref{Xxixit}, we have
\begin{align}\label{upbound}
|X_\xi(\xi,t)|\leq 1+M_1M_\infty t_1=1+||m_0||_{C}||m_0||_{L^1}t_1~ \textrm{ for }~t\in[0,t_1].
\end{align}

Differentiating  \eqref{Xxixit}  with respect to $\xi$ shows that
\begin{align}\label{Xxixixit}
&X_{\xi\xi}(\xi,t)=1+2m_0'(\xi)\int_0^t\bigg(\int_\xi^L G(X(\xi,s)-X(\theta,s))m_0(\theta)d\theta\nonumber\\
&\qquad\qquad -\int_{-L}^\xi G(X(\xi,s)-X(\theta,s))m_0(\theta)d\theta\bigg)ds-2m_0^2(\xi)t\nonumber\\
&\qquad\qquad +2m_0(\xi)\int_0^t X_\xi(\xi,s) \int_{-L}^L G(X(\xi,s)-X(\theta,s))m_0(\theta)d\theta ds.
\end{align}
Hence, we obtain $X_{\xi\xi}\in C(U_{t_1})$ and
$$|X_{\xi\xi}(\xi,t)|\leq 1+2||m_0||_{C^1}||m_0||_{L^1}t_1+2||m_0||_{C}^2t_1+2||m_0||_C^2||m_0||_{L^1}^2t_1^2.$$
 We have $X\in C^2_1(U_{t_1})$.

Similarly, taking  derivative about $\xi$  for $k$ times on both sides of \eqref{Xxixit} gives that
$$X\in C^{k+1}_1(U_{t_1})$$
and \eqref{derivativeproperties} holds.

\end{proof}

\begin{remark}
Monotonicity of $X(\cdot ,t)$ plays an important role in our proof. Without monotonicity, the vector field for the Lagrange dynamics may be not Lipschitz.
From \eqref{Xxixit}, we know $\mathrm{supp}\{X_\xi(\cdot ,t)-1\}\subset(-L,L)$. Hence, we can continuously extend $X_\xi(\cdot ,t)$ globally as
$$X_\xi(\xi,t)=1~\textrm{ for }~\xi\in\mathbb{R}\setminus[-L,L].$$
\end{remark}

\subsection{Classical solutions to the mCH equation}\label{section2}
Next, we prove the short time existence and uniqueness of the classical solutions to \eqref{mCH}-\eqref{initial m}.

The following lemma shows that we can construct  classical solutions to the mCH equation \eqref{mCH}-\eqref{initial m} from the solutions to the Lagrange dynamics \eqref{Lagran dynamics}. Moreover, we show that  the support of $m(\cdot ,t)$ will not change.
\begin{lemma}\label{lagransolution}
Let $m_0\in C_c^k(-L,L)$ for some interger $k\geq1$. Assume that $X\in C_1^{k+1}(U_\delta)$ (for some $\delta>0$)  is the solution of \eqref{Lagran dynamics} and strictly monotonic about $\xi$ for any fixed time $t\in [0,\delta]$. $u(x,t)$, $m(x,t)$ are defined by \eqref{eq:umdef}. And assume $u\in C_1^{k+2}(\mathbb{R}\times[0,\delta])$. Then, $(u(x,t),m(x,t))$ is a classical solution of \eqref{mCH}-\eqref{initial m}.

Moreover, we have
\begin{equation}\label{support of m}
\emph{supp}\{m(\cdot ,t)\}\subset(-L,L),\quad t\in[0,\delta].
\end{equation}
\end{lemma}
\begin{proof}
We denote $(\phi,\psi):=\int_{\mathbb{R}}\phi(x)\psi(x)dx.$ For any test function $\phi\in C_c^\infty(\mathbb{R})$, we have
\begin{align*}
(\phi,m)&=\int_{\mathbb{R}}\phi(x)\int_{-L}^Lm_0(\theta)\delta(x-X(\theta,t))d\theta dx=\int_{-L}^Lm_0(\theta)\phi(X(\theta,t))d\theta.
\end{align*}
\begin{align*}
(\phi, m_t)&=\frac{d}{dt}(\phi,m)=\int_{-L}^Lm_0(\theta)\phi'(X(\theta,t))\dot{X}(\theta,t)d\theta\\
&=\int_{-L}^Lm_0(\theta)\phi'(X(\theta,t))U(X(\theta,t),t)d\theta=\int_{\mathbb{R}}\phi'(x)U(x,t)m(x,t)dx\\
&=-\int_{\mathbb{R}}\phi(x)(U(x,t)m(x,t))_xdx.
\end{align*}
Since that $\phi$ is arbitrary, we have $$m_t+(Um)_x=m_t+[(u^2-u_x^2)m]_x=0.$$

Next, we prove \eqref{support of m}. Because $X(\xi,t)$ is monotonic and $G'(x)=-G(x)$ for $x>0$, we obtain
$$u_x(X(L,t),t)=\int_{-L}^LG'(X(L,t)-X(\theta,t))m_0(\theta)d\theta=-u(X(L,t),t).$$
Hence, we have
$$\dot{X}(L,t)=u^2(X(L,t),t)-u_x^2(X(L,t),t)=0~\textrm{ for }~t\in[0,\delta],$$
which implies
$$X(L,t)\equiv X(L,0)=L.$$
Similarly, we have
$X(-L,t)\equiv X(-L,0)=-L.$

For any $\phi\in C^\infty_c(\mathbb{R})$,  $\mathrm{supp}\{\phi\}\subset\mathbb{R}\setminus(-L,L)$ gives
\begin{equation*}
(\phi,m)=\int_{-L}^Lm_0(\theta)\phi(X(\theta,t))d\theta=0.
\end{equation*}
Hence, \eqref{support of m} holds.
\end{proof}

\begin{remark}
Consider the following general equation with $\alpha>0$,
\begin{align}\label{eq:alphaequation}
m_t+[m(u^2-\alpha^2u_x^2)]_x=0.
\end{align}
When $\mathrm{supp}\{m_0\}\subset(-L,L)$, the support of the classical solution $m(x,t)$ to \eqref{eq:alphaequation} is also contained in $(-L,L)$.
Indeed, by scaling $\tilde{u}(x,t)=u(\alpha x,\alpha t)$ and $\tilde{m}(x,t)=m(\alpha x,\alpha t)=\tilde{u}(x,t)-\tilde{u}_{xx}(x,t)$, $\tilde{u}$ and $\tilde{m}$ satisfy
$$\tilde{m}_t+[(\tilde{u}^2-\tilde{u}_x^2)\tilde{m}]_x=0.$$
Due to $\mathrm{supp}\{\tilde{m}_0\}\subset(-\alpha L,\alpha L)$, by \eqref{support of m} we know $\mathrm{supp}\{\tilde{m}(\cdot,t)\}\subset(-\alpha L,\alpha L)$. Hence, we have
$\mathrm{supp}\{m(\cdot,t)\}\subset(-L,L)$.
\end{remark}

Next, we present a useful lemma which is similar to Lemma \ref{continuous lemma0}.
\begin{lemma}\label{continuous lemma}
Assume $g\in C(U_{t_1})$  and $g(\cdot ,t)\in C_c(-L,L)$ for any fixed time $t\in[0,t_1]$. Let $X\in C_1^1(U_{t_1})$ satisfy \eqref{Xxiproperties} for some constants $C_2>C_1>0$. Set
$$A(x,t):=\int_{-L}^L\delta(x-X(\theta,t))g(\theta,t)d\theta.$$
Then, we have $A\in C(\mathbb{R}\times[0,t_1])$  and
$$\int_{-L}^LG''(x-X(\theta,t))g(\theta,t)d\theta\in C(\mathbb{R}\times[0,t_1]).$$
\end{lemma}
\begin{proof}
From the proof of Lemma \ref{lagransolution}, we know $[X(-L,t),X(L,t)]=[-L,L]$. However, in order to make no confusion, we still use $[X(-L,t),X(L,t)]$ in this proof.

By using the inverse function theorem, for any $t\in[0,t_1]$, there is a continuously differentiable function $Z(\cdot ,t)\in C^1[X(-L,t),X(L,t)]$  such that
$$Z(X(\theta,t),t)=\theta \textrm{ for }\theta\in[-L,L]$$
and
$$X(Z(x,t),t)=x \textrm{ for }x\in[X(-L,t),X(L,t)].$$
Moreover, we have
$$\frac{1}{C_2}\leq Z_x(x,t)\leq\frac{1}{C_1}.$$

Changing variable and using the property of Dirac measure, we have
\begin{align}\label{another representation of m}
A(x,t)&=\int_{-L}^L\delta(x-X(\theta,t))g(\theta,t)d\theta=\int_{X(-L,t)}^{X(L,t)}\delta(x-y)g(Z(y,t),t)Z_x(y,t)dy\nonumber\\
&=\left\{
         \begin{array}{ll}
           0,\quad \textrm{ for }x> X(L,t)\textrm{ or }x< X(-L,t); \\
           g(Z(x,t),t)Z_x(x,t),\quad \textrm{ for }x\in[X(-L,t),X(L,t)].
         \end{array}
       \right.
\end{align}
Next, we separate the proof into three parts, which is similar to the proof of Lemma \ref{continuous lemma0}.

\emph{Step 1}. Continuity at $(x,t)\in\mathbb{R}\times[0,t_1]$ when $x>X(L,t)$. Then case for  $x<X(-L,t)$ is similar.

In this case, we have $A(x,t)=0$.
For any $(y,s)$ closed to $(x,t)$ and because $X\in C(U_{t_1})$, we can assume $y\geq X(L,s)$. Because $g(\cdot ,s)\in C_c(-L,L)$, we have $A(y,s)=0$. Hence, $A$ is continuous at $(x,t)$.

\emph{Step 2}. Continuity at $(x,t)\in\mathbb{R}\times[0,t_1]$ when $x=X(\xi,t)$ for some $\xi\in(-L,L)$. This means $x\in(X(-L,t),X(L,t))$.

Due to the continuity of $X$, for $(y,s)$ closed enough to $(x,t)$, we can assume $y\in [X(-L,s),X(L,s)]$. In other words, there exists $\eta\in[-L,L]$ such that $X(\eta,s)=y$. Because
\begin{align*}
\quad|A(y,s)-A(x,t)|=|g(Z(x,t),t)Z_x(x,t)-g(Z(y,s),s)Z_x(y,s)|,
\end{align*}
we only have to prove $Z$ and $Z_x$ are continuous at $(x,t)$.
 \eqref{will be used} shows that
\begin{align}\label{one}
|Z(x,t)-Z(y,s)|=|\xi-\eta|\leq \frac{1}{C_1}(|x-y|+|X(\xi,t)-X(\xi,s)|),
\end{align}
which means $Z$ is continuous at $(x,t)$.

Because $Z_x(x,t)=\frac{1}{X_\xi(\xi,t)}$ and $Z_x(y,s)=\frac{1}{X_\xi(\eta,s)}$,  we have
\begin{align*}
|Z_x(x,t)-Z_x(y,s)|&=\bigg|\frac{1}{X_\xi(\xi,t)}-\frac{1}{X_\xi(\eta,s)}\bigg|\leq\frac{1}{C_1^2}|X_\xi(\xi,t)-X_\xi(\eta,s)|.
\end{align*}
From \eqref{one} we can see $(\eta,s)\rightarrow(\xi,t)$ as $(y,s)\rightarrow(x,t)$. Together with $X\in C_1^1(U_{t_1})$ implies the continuity of $Z_x(x,t)$ at $(x,t)$.

Hence,  $A(x,t)$ is continuous at $(x,t)$.

\emph{Step 3}. $x=X(L,t)$. The case $x=X(-L,t)$ is similar.

For $(y,s)$ closed to $(x,t)$, we have two cases. When $y>X(L,s)$, we can use Step 1. When there exists $\xi\in(-L,L)$ such that $y=X(\xi,s)$, we can use Step 2.

Put Step 1,2,3 together and we can see $A\in C(\mathbb{R}\times [0,t_1])$.

At last, because $G(x)$ is fundamental solution for Helmholtz operator $1-\partial_{xx}$, we have
\begin{align*}
\int_{-L}^LG''(x-X(\theta,t))g(\theta,t)d\theta=\int_{-L}^LG(x-X(\theta,t))g(\theta,t)d\theta-\int_{-L}^L\delta(x-X(\theta,t))g(\theta,t)d\theta.
\end{align*}
Hence, $\int_{-L}^LG''(x-X(\theta,t))g(\theta,t)d\theta\in C(\mathbb{R}\times[0,t_1]).$

\end{proof}

Now we prove that $u(x,t),m(x,t)$ defined by \eqref{eq:umdef} is a unique classical solution of \eqref{mCH}-\eqref{initial m}.
\begin{theorem}\label{regular solution}
Assuming $m_0\in C_c^k(-L,L)$ $(k\in \mathbb{N}, k\geq1)$. Then, for 
$$t_1<\frac{1}{2||m_0||_{L^1}^2+||m_0||_{L^1}||m_0||_{L^\infty}},$$ 
$u$ given by \eqref{eq:umdef} belongs to $C^{k+2}_1(\mathbb{R}\times[0,t_1])$ and $m$ belongs to $C^k_1(\mathbb{R}\times[0,t_1])$. $(u(x,t),m(x,t))$ is a unique classical solution to \eqref{mCH}-\eqref{initial m}.
\end{theorem}
\begin{proof}
Let $M_1:=||m_0||_{L^1}$ and $M_\infty:=||m_0||_{L^\infty}$.
For $t_1<\frac{1}{2M_1^2+M_1M_\infty}$, by Theorem \ref{noregularityresults}, we know there exist a solution $X\in C_1^{k+1}(U_{t_1})$ to \eqref{Lagran dynamics} satisfying \eqref{Xxiproperties} for $C_1,C_2$ given by \eqref{C_2} and \eqref{C_1}.

\textbf{Part I. Regularity.}

\emph{Step 1.}
When $k=1$, we have $X\in C_1^{2}(U_{t_1})$ and we prove  $u\in C^3_1(\mathbb{R}\times[0,t_1])$.

Taking derivative about $t$ for $u(x,t)$ in \eqref{eq:umdef} gives that
\begin{align*}
\partial_tu(x,t)=-\int_{-L}^LU(X(\theta,t),t)G'(x-X(\theta,t))m_0(\theta)d\theta.
\end{align*}
Because $m_0(\theta)U(X(\theta,t),t)\in C(U_{t_1})$  and $m_0(\cdot )U(X(\cdot ,t),t)\in C_c(-L,L)$ for any fix time $t\in[0,t_1]$, Lemma \ref{continuous lemma} shows that $\partial_tu\in C(\mathbb{R}\times[0,t_1])$.

For the spatial variable $x$, integration by parts leads to
\begin{align*}
u_x(x,t)=\int_{-L}^LG'(x-X(\theta,t))m_0(\theta)d\theta=\int_{-L}^L G(x-X(\theta,t))\partial_\theta\bigg(\frac{m_0(\theta)}{X_\theta(\theta,t)}\bigg)d\theta,
\end{align*}
\begin{align*}
u_{xx}(x,t)&=\int_{-L}^LG'(x-X(\theta,t))\partial_\theta\bigg(\frac{m_0(\theta)}{X_\theta(\theta,t)}\bigg)d\theta,
\end{align*}
and
\begin{align}\label{uxxx}
u_{xxx}(x,t)=\int_{-L}^LG''(x-X(\theta,t))\partial_\theta\bigg(\frac{m_0(\theta)}{X_\theta(\theta,t)}\bigg)d\theta.
\end{align}
Set
$g(\theta, t):=\partial_\theta\bigg(\frac{m_0(\theta)}{X_\theta(\theta,t)}\bigg).$ Then,
$g(\theta,t)$ satisfies the assumption of Lemma \ref{continuous lemma}.
Hence
$$u_{xxx}\in C(\mathbb{R}\times[0,t_1])~\textrm{ and }~u\in C^3_1(\mathbb{R}\times[0,t_1]).$$

\emph{Step 2.} When $k=2$, we have $X\in C_1^{3}(U_{t_1})$.
Integration by parts changes \eqref{uxxx}  into
\begin{align*}
u_{xxx}(x,t)=\int_{-L}^LG'(x-X(\theta,t))\partial_\theta\bigg(\frac{1}{X_\theta}\partial_\theta\bigg(\frac{m_0(\theta)}{X_\theta(\theta,t)}\bigg)\bigg)d\theta.
\end{align*}
Hence
$$\partial_x^4u(x,t)
=\int_{-L}^LG''(x-X(\theta,t))\partial_\theta\bigg(\frac{1}{X_\theta}\partial_\theta\bigg(\frac{m_0(\theta)}{X_\theta(\theta,t)}\bigg)\bigg)d\theta.$$
And Lemma \ref{continuous lemma} shows that $u\in C^4_1(\mathbb{R}\times[0,t_1])$.

\emph{Step 3.}
If $k>2$, we can keep using integration by parts and Lemma \ref{continuous lemma}  and  obtain
$$u\in C^{k+2}_1(\mathbb{R}\times[0,t_1]).$$

\emph{Step 4.} Because $m=u-u_{xx}$, from the above steps, we already know $\partial_x^km\in C(\mathbb{R}\times[0,t_1])$. In this step, we show $\partial_tm\in C(\mathbb{R}\times[0,t_1]).$ Due to \eqref{support of m}, we only have to show $\partial_t m\in C([-L,L]\times[0,t_1])$. 
From \eqref{another representation of m}, for $x\in(-L,L)$ and  $X(\xi,t)=x$, we have
\begin{align}\label{m1}
m(X(\xi,t),t)=m_0(Z(x,t))Z_x(x,t)=\frac{m_0(\xi)}{X_\xi(\xi,t)}.
\end{align}
Taking derivative of both sides of  \eqref{m1}, we have
\begin{align}\label{m2}
\frac{d}{dt}m(X(\xi,t),t)&=m_x(X(\xi,t),t)X_t(\xi,t)+\partial_tm(X(\xi,t),t)\nonumber\\
&=[m_x(u^2-u_x^2)](x,t)+m_t(x,t),
\end{align}
and
\begin{align}\label{m3}
\frac{d}{dt}\frac{m_0(\xi)}{X_\xi(\xi,t)}=-\frac{2m_0(\xi)mu_x(X(\xi,t),t)}{X_\xi(\xi,t)}=-2m^2u_x(x,t).
\end{align}
Combining \eqref{m1}, \eqref{m2} and \eqref{m3}, we obtain
$$m_t=-[m(u^2-u_x^2)]_x\in C([-L,L]\times[0,t_1]).$$

From the above proof (or Lemma \ref{lagransolution}), we can see that $u(x,t),m(x,t)$ is a classical solution to \eqref{mCH}-\eqref{initial m}.

\textbf{Part II. Uniqueness of the classical solution to \eqref{mCH}-\eqref{initial m}}.

Assume there is another classical solution $m_1\in C_1^k(\mathbb{R}\times[0,t_1])$ to \eqref{mCH}-\eqref{initial m}. $u_1=G\ast m_1\in  C_1^{k+2}(\mathbb{R}\times[0,t_1]).$ We prove that $u_1(x,t)$ can also be defined by the solution $X(\xi,t)$ to \eqref{Lagran dynamics}, which means
\begin{align}\label{u1=u}
u_1(x,t)=\int_{-L}^LG(x-X(\theta,t))m_0(\theta)d\theta=u(x,t).
\end{align}
To this end, define another characteristics $Y(\xi,t)$ by
\begin{align*}
\dot{Y}(\xi,t)=(u_1^2-\partial_xu_{1}^2)(Y(\xi,t),t),
\end{align*}
subject to
$$Y(\xi,0)=\xi\in\mathbb{R}.$$
By standard ODE theory, we can obtain a solution $Y\in C_1^{k+1}(\mathbb{R}\times[0,t_1])$.

\emph{Step 1.} We prove
$$u_1(x,t)=\int_{-L}^LG(x-Y(\theta,t))m_0(\theta)d\theta.$$

Taking derivative with respect to $\xi$ shows that
\begin{align}\label{unique1}
\dot{Y}_{\xi}(\xi,t)=2(m_1\partial_xu_1)(Y(\xi,t),t)Y_\xi(\xi,t).
\end{align}
Taking time derivative of $m_1(Y(\xi,t),t)Y_\xi(\xi,t)$ gives that
\begin{align*}
\quad\frac{d}{dt}[m_1(Y(\xi,t),t)Y_\xi(\xi,t)]&=[\partial_tm_1(Y,t)+\partial_xm_1(Y,t)Y_t]Y_\xi+m_1(Y,t)Y_{\xi t}\nonumber\\
&=[\partial_tm_1+(u_1^2-\partial_xu_1^2)\partial_xm_1]Y_\xi+2\partial_xu_1m_1^2Y_\xi\nonumber\\
&=[\partial_tm_1+[(u_1^2-\partial_xu_1^2)m_1]_x]Y_\xi=0.
\end{align*}
This implies
\begin{align}\label{u11}
m_1(Y(\theta,t),t)Y_\xi(\theta,t)=m_0(\theta),~\textrm{ for }~\theta\in[-L,L].
\end{align}
Hence, we can see
\begin{align}\label{u1}
u_1(x,t)=\int_{\mathbb{R}}G(x-y)m_1(y,t)dy&=\int_\mathbb{R}G(x-Y(\theta,t))m_1(Y(\theta,t),t)Y_\xi(\theta,t)d\theta\nonumber\\
&=\int_{-L}^LG(x-Y(\theta,t))m_0(\theta)d\theta.
\end{align}

\emph{Step 2.} We prove $Y(\xi,t)=X(\xi,t)$.

From \eqref{u1}, we obtain
\begin{align*}
&\quad\dot{Y}(\xi,t)=(u_1^2-\partial_xu_1^2)(Y(\xi,t),t)\\
&=\bigg(\int_{-L}^LG(Y(\xi,t)-Y(\theta,t))m_0(\theta)d\theta\bigg)^2-\bigg(\int_{-L}^LG'(Y(\xi,t)-Y(\theta,t))m_0(\theta)d\theta\bigg)^2,
\end{align*}
which means that $Y(\xi,t)$ is also a solution to \eqref{Lagran dynamics}.

From Theorem \ref{noregularityresults}  we know that the strictly monotonic solution to \eqref{Lagran dynamics} is unique. Therefore, to prove $Y(\xi,t)=X(\xi,t)$, we only have to prove $Y(\cdot ,t)$ is strictly monotonic for $t\in[0,t_1]$.

Combining \eqref{unique1} and \eqref{u11}  gives that
\begin{equation*}
Y_\xi(\xi,t)=\exp\bigg(2\int_0^t(m_1\partial_xu_1)(Y(\xi,s),s)ds\bigg),\quad (\xi,t)\in [-L,L]\times[0,t_1].
\end{equation*}
Because $||Y||_{L^\infty([-L,L]\times[0,t_1])}<+\infty$, $u_1\in C^{k+2}_1(\mathbb{R}\times[0,t_1])$ and $m_1\in C^k_1(\mathbb{R}\times[0,t_1])$, the minimum and maximum of $(m_1\partial_xu_1)(Y(\xi,s),s)$ can be obtained on $[-L,L]\times[0,t_1]$.
Hence
\begin{align*}
e^{2K_1t_1}\leq Y_\xi(\xi,t)\leq e^{2K_2t_1},~\textrm{ for }~t\in[0,t_1],
\end{align*}
where
$$K_1=\min_{(\xi,s)\in [-L,L]\times[0,t_1]}(m_1\partial_xu_1)(Y(\xi,s),s)$$
and
$$K_2=\max_{(\xi,s)\in [-L,L]\times[0,t_1]}(m_1\partial_1u_1)(Y(\xi,s),s).$$
Hence, $Y(\cdot ,t)$ is strictly monotonic for $t\in[0,t_1]$.

Combining Step 1 and Step 2, we obtain \eqref{u1=u}.
\end{proof}

\begin{remark}
\eqref{u11} also can be easily obtained by \cite[Theorem 5.34]{Villani}

The strictly monotonic property of $X$ plays an crucial role in the proof of the above Theorem. Whenever $X$ is strictly monotonic, we can use integration by parts to obtain the regularity of $u(x,t)$. Conversely, if $m(x,t)$ is a classical solution, then the characteristics for the mCH equation is strictly monotonic.
\end{remark}

For the convenience of the rest proof, we summarize the results in the proof of Part II of Theorem \ref{regular solution} and give a corollary.
\begin{corollary}\label{increasing}
Let $m_0\in C_c^k(-L,L)$ $(k\in \mathbb{N}, k\geq1)$ and
$X\in C^{k+1}_1([-L,L]\times[0,T])$
be the solution to \eqref{Lagran dynamics}. $u\in C^{k+2}_1(\mathbb{R}\times[0,T])$, $m\in C_1^k(\mathbb{R}\times[0,T])$ is a classical solution to \eqref{mCH}-\eqref{initial m}.  Then, we have
\begin{equation}\label{Xincrease}
X_\xi(\xi,t)=\exp\bigg(2\int_0^t(mu_x)(X(\xi,s),s)ds\bigg)~\textrm{ for }~(\xi,t)\in [-L,L]\times[0,T]
\end{equation}
and
\begin{align}\label{Xderivativebound}
e^{2K_1T}\leq X_\xi(\xi,t)\leq e^{2K_2T}~\textrm{ for }~(\xi,t)\in [-L,L]\times[0,T],
\end{align}
where
$$K_1=\min_{(\xi,s)\in [-L,L]\times[0,T]}(mu_x)(X(\xi,s),s)$$
and
$$K_2=\max_{(\xi,s)\in [-L,L]\times[0,T]}(mu_x)(X(\xi,s),s).$$

Moreover, we have
\begin{align}\label{keepsign}
m(X(\xi,t),t)X_\xi(\xi,t)=m_0(\xi)~\textrm{ for }~(\xi,t)\in(-L,L)\times[0,T].
\end{align}

\end{corollary}
\begin{proof}
The proof for \eqref{Xderivativebound} and \eqref{keepsign} is the same as the proof for uniqueness in Theorem \ref{regular solution}.
\end{proof}
\begin{remark}\label{positivenegative}
From \eqref{keepsign}, we know that $m(X(\theta,t),t)$ does not change sign for any $t\in[0,T]$. We present a precise argument here.

 Set
$$A^+:=\{\xi\in (-L,L):m_0(\xi)>0\},~~A^-:=\{\xi\in (-L,L):m_0(\xi)<0\},$$
and
$$A^0:=\{\xi\in(-L,L):m_0(\xi)=0\}.$$
Hence,
$$A^+\cup A^-\cup A^0=(-L,L).$$
For $t\in[0,T]$, denote
$$A^+_t:=\{X(\xi,t)\in\mathbb{R}:\xi\in A^+\},~~A^-_t:=\{X(\xi,t)\in\mathbb{R}:\xi\in A^-\},$$
and
$$A^0_t:=\{X(\xi,t)\in\mathbb{R}:\xi\in A^0\}.$$
Then, we have $A^+_0=A^+$, $A^-_0=A^-$ and $A^0_0=A^0$.  Due to the monotonicity of $X(\cdot,t)$, one can easily show that $A_t^+$ and $A_t^-$ are open sets while $A^0$ is a closed set for $t\in[0,T]$.
Also we have
$$A_t^+\cup A_t^-\cup A^0_t=(X(-L,t),X(L,t))$$
and (by \eqref{keepsign})
\begin{align*}
m(x,t)\left\{
          \begin{array}{ll}
            >0,~\textrm{ for }~ x\in A_t^+ \\
            =0,~\textrm{ for }~ x\in A_t^0 \\
            <0,~\textrm{ for }~ x\in A_t^-.
          \end{array}
        \right.
\end{align*}

Due to
$$\dot{X}_\xi(\xi,t)=2(mu_x)(X(\xi,t),t)\equiv0~\textrm{ for }~\xi\in A^0,$$
we obtain
\begin{align*}
X_\xi(\xi,t)\equiv X_\xi(\xi,0)=1~\textrm{ for }~\xi\in A^0,~~t\in[0,T].
\end{align*}
This can also be obtained by \eqref{Xxixit}.
\end{remark}

\subsection{Solution extension}\label{section3}
In this subsection, we will show that as long as classical solutions to \eqref{mCH}-\eqref{initial m} satisfying $||m(\cdot,t)||_{L^\infty}<\infty$ we can extend the solutions $X$ and $m$ in time.

\begin{proposition}\label{extendsolution}
Assume $m_0\in C_c^k(-L,L)$ and $X\in C_1^{k+1}([-L,L]\times[0,T_0))$ is the solution to \eqref{Lagran dynamics}. Let $m\in C_1^k(\mathbb{R}\times[0,T_0))$ be the corresponding solution to \eqref{mCH}-\eqref{initial m}. If
$$\sup_{t\in[0,T_0)}||m(\cdot ,t)||_{L^\infty}<+\infty,$$
then there exists $\widetilde{T}_0>T_0$ such that
$$X\in C_1^{k+1}([-L,L]\times[0,\widetilde{T}_0])$$
is a solution to \eqref{Lagran dynamics},
and
$$u\in C^{k+2}_1(\mathbb{R}\times[0,\widetilde{T}_0]),\quad m\in C_1^k(\mathbb{R}\times[0,\widetilde{T}_0])$$
is a solution to \eqref{mCH}-\eqref{initial m}.
\end{proposition}
\begin{proof}
There exists a constant $\widetilde{M}_\infty$ satisfies
$$\sup_{t\in[0,T_0)}||m(\cdot ,t)||_{L^\infty}\leq \widetilde{M}_\infty.$$
From Lemma \ref{lagransolution}, we know $m(\cdot ,t)$ has a uniform (in $t$) support. Hence, there exists a constant $\widetilde{M_1}$ such that
$$\sup_{t\in[0,T_0)}||m(\cdot ,t)||_{L^1}\leq \widetilde{M}_1.$$

Consider time $T_1=T_0-\frac{1}{3(2\widetilde{M}_1^2+\widetilde{M}_1\widetilde{M}_\infty)}$. Our target is to prove that the classical solution can be extend to $\widetilde{T}_0:=T_1+\frac{1}{2(2\widetilde{M}_1^2+\widetilde{M}_1\widetilde{M}_\infty)}>T_0$. We will show this in two steps.

\emph{Step 1}. In this step we consider a dynamic system from time $T_1.$

From \eqref{support of m} we know $m(\cdot ,T_1)\in C_c^k(-L,L)$.
Set
$$\widetilde{m}_0(\widetilde{\theta}):=m(\widetilde{\theta},T_1)~\textrm{ for }~\widetilde{\theta}\in[-L,L].$$
Consider dynamics for $\widetilde{X}(\widetilde{\xi},t)$:
\begin{align}\label{T1initialdynamics}
\left\{
  \begin{array}{ll}
    \displaystyle{\frac{d}{dt}\widetilde{X}(\widetilde{\xi},t)=\bigg(\int_{-L}^{L}G(\widetilde{X}(\widetilde{\xi},t))-\widetilde{X}(\widetilde{\theta},t))
\widetilde{m}_0(\widetilde{\theta})d\widetilde{\theta}\bigg)^2}  \\ \qquad\qquad\qquad\qquad\qquad\qquad
\displaystyle{-\bigg(\int_{-L}^{L}G'(\widetilde{X}(\widetilde{\xi},t))-\widetilde{X}(\widetilde{\theta},t))\widetilde{m}_0(\widetilde{\theta})d\widetilde{\theta}\bigg)^2,} \\
    \widetilde{X}(\widetilde{\xi},0)=\widetilde{\xi}\in [-L,L].
  \end{array}
\right.
\end{align}
Because $ \widetilde{m}_0(\cdot )=m(\cdot ,T_1)\in C^k_c(-L,L)$, by Theorem \ref{regular solution}, we know that for any $$0<t_1< \frac{1}{2\widetilde{M}_1^2+\widetilde{M}_1\widetilde{M}_\infty},$$
there exists a solution
$\widetilde{X}(\widetilde{\xi},t)$ to \eqref{T1initialdynamics} and a classical solution ($\widetilde{u}(x,t),\widetilde{m}(x,t)$) to \eqref{mCH} subject to initial condition
$$\widetilde{m}(x,0)=\widetilde{m}_0(x)=m(x,T_1).$$
Moreover,
$$\widetilde{X}\in C^{k+1}_1([-L,L]\times[0,t_1]),$$
$$\widetilde{u}\in C^{k+2}_1(\mathbb{R}\times[0,t_1])~\textrm{ and }~\widetilde{m}\in C^k_1(\mathbb{R}\times[0,t_1]).$$
Choose $t_1=\frac{1}{2(2\widetilde{M}_1^2+\widetilde{M}_1\widetilde{M}_\infty)}$ and  set $\widetilde{T}_0=T_1+t_1.$ Thus $T_0<\widetilde{T}_0$.

\emph{Step 2.} In this step we extend the solutions to $[0,\widetilde{T}_0]$.

Changing variable by $\widetilde{\xi}=X(\xi,T_1)$, initial value $\widetilde{X}\big(X(\xi,T_1),0\big)=X(\xi,T_1)$ allows us to define
\begin{align}\label{Xextend}
X(\xi,T_1+t):=\widetilde{X}\big(X(\xi,T_1),t\big)~\textrm{ for }~\xi\in[-L,L],t\in[0,t_1]
\end{align}
and
we have
$$X\in C_1^{k+1}([-L,L]\times[0,\widetilde{T}_0]).$$
Similarly, because $\widetilde{m}(x,0)=m(x,T_1)$, we can use $\widetilde{u}(x,t),\widetilde{m}(x,t)$ to define
\begin{align*}
u(x,T_1+t):=\widetilde{u}(x,t),\quad m(x,T_1+t):=\widetilde{m}(x,t)\quad\textrm{for }(x,t)\in\mathbb{R}\times[0,t_1]
\end{align*}
and we have
$$u\in C_1^{k+2}(\mathbb{R}\times[0,\widetilde{T}_0]),\quad m\in C_1^k(\mathbb{R}\times[0,\widetilde{T}_0]).$$
Moreover, we can see $(u(x,t),m(x,t))$ we defined is a classical solution to \eqref{mCH}-\eqref{initial m} in $[0,\widetilde{T}_0]$.

Next, we show $X(\xi,t)$ satisfies \eqref{Lagran dynamics} in $[0,\widetilde{T}_0]$.

Actually, changing variable by $\widetilde{\theta}=X(\theta,T_1)$ and combining \eqref{Xextend} and \eqref{keepsign} lead to
\begin{align*}
u(x,T_1+t)&=\widetilde{u}(x,t)=\int_{-L}^{L}G(x-\widetilde{X}(\widetilde{\theta},t))\widetilde{m}_0(\widetilde{\theta})d\widetilde{\theta}\\
&=\int_{-L}^LG(x-X(\theta,T_1+t))m(X(\theta,T_1),T_1+t)X_\theta(\theta,T_1+t)d\theta\\
&=\int_{-L}^LG(x-X(\theta,T_1+t))m_0(\theta)d\theta.
\end{align*}
Similarly,
$$\int_{-L}^{L}G'(x-\widetilde{X}(\widetilde{\theta},t))\widetilde{m}_0(\widetilde{\theta})d\widetilde{\theta}=u_x(x,T_1+t).$$
Therefore, \eqref{T1initialdynamics} turns into
\begin{align*}
\left\{
  \begin{array}{ll}
    \dot{X}(\xi,T_1+t)=u^2(X(\xi,T_1+t),T_1+t)-u_x^2(X(\xi,T_1+t),T_1+t),\\
    X(\xi,T_1+0)=\widetilde{X}(X(\xi,T_1),0)=X(\xi,T_1),
  \end{array}
\right.
\end{align*}
for $\xi\in[-L,L]$ and $t\in[0,t_1]$.

Hence, $X\in C_1^{k+1}([-L,L]\times[0,\widetilde{T}_0])$ is a solution to \eqref{Lagran dynamics}. Corollary \ref{increasing} ensures the strictly monotonicity of $X(\cdot ,t)$ for $t\in [0,\widetilde{T}_0]$. Therefore, $X(\xi,t)$ is the unique solution which extends the solution to $\widetilde{T}_0.$

\end{proof}

\section{Blow-up criteria}\label{sec2}
In this section, we give some criteria on finite time blow-up of classical solutions to the mCH equation.

Let $T_{max}>0$ be the maximal existence time of classical solution to the mCH equation.  In other words, $T_{max}$ satisfies
\begin{gather*}
\left\{
\begin{split}
&||m(\cdot ,t)||_{L^\infty}<+\infty,\quad0\leq t<T_{max}, \\
&\limsup_{t\rightarrow T_{max}}||m(\cdot ,t)||_{L^\infty}=+\infty.
\end{split}
\right.
\end{gather*}

Next lemma shows that the solution to Lagrange dynamics \eqref{Lagran dynamics} can be extended to the blow-up time $T_{max}$.
\begin{lemma}\label{XTmaxlemma}
Let $m_0\in C_c^k(-L,L)$. Let $T_{max}$ be the maximal existence time for the classical solution $m(x,t)$ to \eqref{mCH}-\eqref{initial m} and $X\in C^{k+1}_1([-L,L]\times[0,T_{max}))$ be the solution to \eqref{Lagran dynamics}. Then we have
\begin{align}\label{Xtmax}
X\in C^{k+1}_1([-L,L]\times[0,T_{max}]).
\end{align}
\end{lemma}
\begin{proof}
Let $t$ go to $T_{max}$ in \eqref{uniqueXT} and we obtain $X(\xi,T_{max})$.  Using \eqref{upbound} and Lipschitz property of $G(x)=\frac{1}{2}e^{-|x|}$, we can obtain that
$$X\in C([-L,L]\times[0,T_{max}]).$$
Let $t$ go to $T_{max}$ in \eqref{Xxixit} and \eqref{Xxixixit}. Similarly, combining \eqref{derivativeproperties} gives
\begin{align*}
X\in C^2_0([-L,L]\times[0,T_{max}]).
\end{align*}
Keep doing like this and we can see
\begin{align*}
X\in C^{k+1}_0([-L,L]\times[0,T_{max}]).
\end{align*}

At last, let $t$ go to $T_{max}$ in \eqref{short representation} and combining \eqref{Lagran dynamics}, we have  $\partial_tX\in C([-L,L]\times[0,T_{max}])$.

\end{proof}

We have the following blow up criteria.
\begin{theorem}\label{criterion}
Let $m_0\in C_c^k(-L,L)$ $(k\in \mathbb{N}, k\geq1)$. $X(\xi,t)$ is the solution to Lagrange dynamics \eqref{Lagran dynamics}. Assume $T_{max}<+\infty$ is the maximum existence time for the classical solution to \eqref{mCH}-\eqref{initial m}. Then, the following equivalent statements hold.\\
$\mathrm{(i)}$
\begin{align}
\limsup_{t\rightarrow T_{max}}||m(\cdot ,t)||_{L^\infty}=+\infty,\label{mblowup}
\end{align}
$\mathrm{(ii)}$
\begin{gather}\label{blowupfor lagrange}
\left\{
  \begin{split}
   &X_\xi(\xi,t)>0~\textrm{ for }~(\xi,t)\in[-L,L]\times[0,T_{max});\\
  &\min_{\xi\in[-L,L]}X_\xi(\xi,T_{max})=0.
  \end{split}
\right.
\end{gather}
$\mathrm{(iii)}$
\begin{align}
\liminf_{t\rightarrow T_{max}}\Big\{\inf_{\xi\in[-L,L]}\int_0^t(mu_x)(X(\xi,s),s)ds\Big\}=-\infty,\label{equal2}
\end{align}
$\mathrm{(iv)}$
\begin{equation}
\liminf_{t\rightarrow T_{max}}\Big\{\inf_{x\in\mathbb{R}}(mu_x)(x,t)\Big\}=-\infty,\label{equal4}
\end{equation}
$\mathrm{(v)}$
\begin{equation}
\limsup_{t\rightarrow T_{max}}||m(\cdot ,t)||_{W^{1,p}}=+\infty~\textrm{ for }p\geq1, \label{equal5}
\end{equation}
$\mathrm{(vi)}$
\begin{align}\label{equal61}
\int_0^{T_{max}}||m(\cdot ,t)||_{L^\infty}dt=+\infty.
\end{align}

\end{theorem}
\begin{proof}
We follow the following lines to prove this theorem,
$$\eqref{mblowup}\Rightarrow\eqref{blowupfor lagrange}\Rightarrow\eqref{equal2}\Rightarrow
\eqref{equal4}\Rightarrow\eqref{equal5}\Rightarrow\eqref{mblowup}$$
and
$$\eqref{equal2}\Rightarrow\eqref{equal61}\Rightarrow\eqref{mblowup}.$$

\emph{Step 1}. We prove $\eqref{mblowup}\Rightarrow\eqref{blowupfor lagrange}$.

 Assume $m(x,t)$ blows up in finite time $T_{max}$. We prove \eqref{blowupfor lagrange} by contradiction.
From Lemma \ref{XTmaxlemma}, we know $X\in C_1^2([-L,L]\times[0,T_{max}])$.
If $\eqref{blowupfor lagrange}$ does not hold, then we have
\begin{align*}
\min\big\{X_\xi(\xi,t):(\xi,t)\in[-L,L]\times[0,T_{max}]\big\}>C_1>0.
\end{align*}
Combining \eqref{keepsign} and \eqref{support of m}, we have
\begin{align*}
\sup_{t\in[0,T_{max})}||m(\cdot ,t)||_{L^\infty(\mathbb{R})}&=\sup_{t\in[0,T_{max})}||m(X(\cdot ,t),t)||_{L^\infty(-L,L)}\\
&=\sup_{t\in[0,T_{max})}\bigg|\bigg|\frac{m_0(\cdot )}{X_\theta(\cdot ,t)}\bigg|\bigg|_{L^\infty(-L,L)}\leq \frac{||m_0||_{L^\infty}}{C_1}.
\end{align*}
This is a contradiction to \eqref{mblowup}.

\emph{Step 2.} We prove $\eqref{blowupfor lagrange}\Rightarrow\eqref{equal2}.$

From \eqref{blowupfor lagrange}, we have
$$\liminf_{t\rightarrow T_{max}}\Big\{\inf_{\xi\in[-L,L]}X_\xi(\xi,t)\Big\}=0.$$
Together with \eqref{Xincrease}, we can see $\eqref{blowupfor lagrange}\Rightarrow\eqref{equal2}$.

\emph{Step 3.} We prove $\eqref{equal2}\Rightarrow\eqref{equal4}$.

 \eqref{equal2} implies that
\begin{align}
\liminf_{t\rightarrow T_{max}}\Big\{\inf_{\xi\in[-L,L]}(mu_x)(X(\xi,t),t)\Big\}=-\infty.\label{equal3}
\end{align}
Because of \eqref{support of m}, for any $t\in[0,T_{max})$ we have
$$\inf_{\xi\in[-L,L]}(mu_x)(X(\xi,t),t)=\inf_{x\in[-L,L]}mu_x(x,t)=\inf_{x\in\mathbb{R}}mu_x(x,t).$$
Hence, we can see that \eqref{equal3} and \eqref{equal4} are equivalent.

\emph{Step 4.} We prove $\eqref{equal4}\Rightarrow\eqref{equal5}$.

Assume \eqref{equal4} holds.
We prove \eqref{equal5} by contradiction. For any $1\leq p\leq+\infty$, if
$$\limsup_{t\rightarrow T_{max}}||m(\cdot ,t)||_{W^{1,p}}<+\infty,$$
then
$$\sup_{t\in[0,T_{max})}||m(\cdot ,t)||_{W^{1,p}}<+\infty.$$
 $W^{1,p}(\mathbb{R})\subset L^\infty(\mathbb{R})$ with continuous injection for all $1\leq p\leq +\infty$ implies that
$$\sup_{t\in[0,T_{max})}||m(\cdot, t)||_{L^\infty}<+\infty.$$
On the other hand, we have
\begin{align}\label{uxbound}
\sup_{t\in[0,T_{max})}||u_x(\cdot,t)||_{L^\infty}\leq\sup_{t\in[0,T_{max})}\bigg|\bigg|\int_{-L}^LG'(\cdot -X(\theta,t))m_0(\theta)d\theta\bigg|\bigg|_{L^\infty}\leq\frac{1}{2}||m_0||_{L^1}.
\end{align}
Hence we obtain $\sup_{t\in[0,T_{max})}||mu_x(\cdot ,t)||_{L^\infty}<+\infty$, which is a contradiction with \eqref{equal4}. Therefore, \eqref{equal5} holds.

\emph{Step 5.} We prove $\eqref{equal5}\Rightarrow\eqref{mblowup}$.

Assume \eqref{equal5} holds.
If $\sup_{t\in[0,T_{max})}||m(\cdot ,t)||_{L^\infty}<+\infty$, by Proposition \ref{extendsolution}, there exists  $T>T_{max}$ such that
$m\in C_1^1(\mathbb{R}\times[0,T]).$ Because $m(\cdot ,t)$ has uniform compact support for $t\in[0,T]$, we have
$$\sup_{t\in[0,T_{max})}||m(\cdot, t)||_{W^{1,p}}\leq\sup_{t\in[0,T]}||m(\cdot,t)||_{W^{1,p}}<+\infty,$$
which is a contradiction.

\emph{Step 6.} At last, we prove $$\eqref{equal2}\Rightarrow\eqref{equal61}\Rightarrow\eqref{mblowup}.$$

When \eqref{equal61} holds, one can easily obtain \eqref{mblowup}. So, we only have to prove $\eqref{equal2}\Rightarrow\eqref{equal61}$.
\eqref{equal2} implies
$$\limsup_{t\rightarrow T_{max}}\Big\{\sup_{x\in\mathbb{R}}\int_0^t|mu_x(x,s)|ds\Big\}=+\infty.$$
Due to \eqref{uxbound}, we obtain
$$\sup_{x\in\mathbb{R}}\int_0^t|mu_x(x,s)|ds\leq C\int_0^t||m(\cdot ,s)||_{L^\infty}ds\leq C\int_{0}^{T_{max}}||m(\cdot ,t)||_{L^\infty}dt$$
and this gives \eqref{equal61}.

\end{proof}

\begin{remark}
\eqref{blowupfor lagrange} shows that  there is a $\xi_0$ such that $X_\xi(\xi_0,T_{max})=0$. This means $T_{max}$ is an onset time of collision of characteristics. Now, we can conclude that if $m(x,t)$ blows up in finite time $T_{max}$, then we have
$$X\in C_1^{k+1}([-L,L]\times[0,T_{max}])~\textrm{ and }~m\in C_1^k(\mathbb{R}\times[0,T_{max})).$$

The blow-up criterion \eqref{equal4} can also be found in \cite{GuiLiuOlverQu}.
Besides, \eqref{equal61} is similar to the well known blow-up criterion for smooth solutions to $3D$ Euler equation \cite{BealeKatoMajda}.
\end{remark}

\begin{remark}[Other equivalent criteria]\label{equivalentcriterionremark}
Because $m(x,t)$ has compact support for $t\in[0,T_{max})$, by Poincar\'e inequality, \eqref{equal5} is equivalent to (for any $1\leq p\leq+\infty$)
\begin{align}\label{equal52}
\limsup_{t\rightarrow T_{max}}||m_x(\cdot ,t)||_{L^p}=+\infty.
\end{align}
Because $m=u-u_{xx}$ and $|u(x,t)|=\Big|\int_{-L}^LG(x-X(\theta,t))m_0(\theta)d\theta\Big|\leq \frac{1}{2}||m_0||_{L^1}$, we know that \eqref{mblowup} is equivalent to
\begin{align*}
\limsup_{t\rightarrow T_{max}}||u_{xx}(\cdot, t)||_{L^\infty}=+\infty.
\end{align*}
\eqref{uxbound} tells us $u_x$ is bounded. Hence the blow up behavior is different with the Camassa-Holm equation, where $u_x$ becomes unbounded \cite{Constantin11,Constantin22}.

When $m_0(x)\geq0$, equality \eqref{keepsign} implies $m(x,t)\geq0$ for any $t\in[0,T_{max})$. Then, all the above blow-up criterions are equivalent to
$$\limsup_{t\rightarrow T_{max}}\Big\{\sup_{x\in \mathbb{R}}m(x,t)\Big\}=+\infty.$$
\end{remark}

Next, when $m_0\in C_c^k(-L,L)$ ($k\geq2$), we give another proof for \eqref{equal61} based on  \eqref{equal52}(p=2).
\begin{proof}[Another proof for \eqref{equal61}]
By Theorem \ref{noregularityresults} and Theorem \ref{regular solution}, we know $m\in C_1^k(\mathbb{R}\times[0,T_{max}))$.
From \eqref{mCH}, we obtain
$$\partial_t(m_x)=-(u^2-u_x^2)(m_x)_x-2u_{xx}m^2-6u_xmm_x.$$
Multiplying both sides by $m_x$  and taking integral show that
$$\frac{1}{2}\frac{d}{dt}\int_{\mathbb{R}}m^2_xdx
=-6\int_{\mathbb{R}}mu_xm_x^2dx-2\int_{\mathbb{R}}m^2m_xu_{xx}dx-\frac{1}{2}\int_{\mathbb{R}}(u^2-u_x^2)(m^2_x)_xdx.$$
Integration by parts for the last term implies that
$$\frac{d}{dt}\int_{\mathbb{R}}m^2_xdx=-10\int_{\mathbb{R}}mu_xm_x^2dx-2\int_{\mathbb{R}}m^2m_xu_{xx}dx.$$
On the other hand, we have
\begin{align*}
\int_{\mathbb{R}}m^2m_xu_{xx}dx&=\int_{\mathbb{R}}m^2m_x(u-m)dx\\
&=\int_{\mathbb{R}}\frac{1}{3}(m^3)_xu-\frac{1}{4}(m^4)_xdx=-\frac{1}{3}\int_{\mathbb{R}}m^3u_xdx.
\end{align*}
Hence
$$\frac{d}{dt}\int_{\mathbb{R}}m^2_xdx=-10\int_{\mathbb{R}}mu_xm_x^2dx+\frac{2}{3}\int_{\mathbb{R}}m^3u_xdx.$$
Inequality \eqref{uxbound} gives
$$\int_{\mathbb{R}}mu_xm_x^2dx\leq C||m||_{L^\infty}\int_{\mathbb{R}}m_x^2dx.$$
By Poincar\'e inequality, we have
\begin{align*}
\int_{\mathbb{R}}m^3u_{x}dx\leq C||m||_{L^\infty}\int_{\mathbb{R}}|m^2|dx\leq C||m||_{L^\infty}\int_{\mathbb{R}}m_x^2dx.
\end{align*}
Hence
$$\frac{d}{dt}\int_{\mathbb{R}}m^2_xdx\leq C||m||_{L^\infty}\int_{\mathbb{R}}m_x^2dx.$$
Gronwall's inequality shows that
\begin{align*}
||m_x||^2_{L^2}\leq ||\partial_xm_0||^2_{L^2}\exp\Big\{C\int_0^t||m||_{L^\infty}ds\Big\}.
\end{align*}
which implies $\eqref{equal52}(p=2)\Rightarrow\eqref{equal61}$.

\end{proof}

\section{Finite time blow up and almost global existence of classical solutions}\label{sec3}

In the rest of this paper, we assume $m_0\in C^1_c(-L,L)$.

In this section, we show that for some initial data  solutions to the mCH equation blow up in finite time. Some blow-up rates  are obtained. Moreover, for any $\epsilon>0$ and initial data $\epsilon m_0(x)\in C^1_c(\mathbb{R})$, we prove that the lifespan of the classical solutions satisfies
$$T_{max}(\epsilon m_0)\sim\frac{C}{\epsilon^2},$$
 where $C$ is a constant depends on $m_0(x).$

Our finite time blow-up results are similar to the blow-up results in \cite{Chenrongming,GuiLiuOlverQu,LiuOlver} but with some subtle differences.  All these three papers apply the idea from transport equation  and  focus on the derivative of $u^2-u_x^2$ which is $2mu_x$.  Comparing with  \cite[Theorem 5.2,5.3]{GuiLiuOlverQu}, we show finite time blow-up for $m_0$ which can change its sign. Besides, our starting point do not have to be the maximum point of $m_0$ in contrast with \cite[Theorem 1.3]{LiuOlver}. The main idea of our proof is similar to \cite[Theorem 1.5]{Chenrongming} which shows blow-up for a sign-changing $m_0$ with the effect of the linear dispersion term $\gamma u_x$ ($\gamma\geq0$). 

We have the following proposition.

\begin{proposition}\label{preparelemma1}
Suppose $m_0\in C_c^1(-L,L)$. Let $T_{max}$ be the maximal time of the existence of the corresponding classical solution $m(x,t)$ to \eqref{mCH}-\eqref{initial m}. $X\in C_1^2([-L,L]\times[0,T_{max}))$ is the solution to \eqref{Lagran dynamics}. \\
$\mathrm{(i)}$
If $\xi_0\in[-L,L]$ satisfies $m_0(\xi_0)\neq0$, then we have
\begin{align}\label{blowupintegration}
X_\xi(\xi_0,t)=1+2m_0(\xi_0)\int_0^t u_x(X(\xi_0,s),s)ds~\textrm{ for }~t\in[0,T_{max}).
\end{align}
$\mathrm{(ii)}$
We have the following lower bound for blow-up time
\begin{align}\label{blowupbelowbound}
T_{max}\geq \frac{1}{||m_0||_{L^\infty}||m_0||_{L^1}}.
\end{align}
\end{proposition}

\begin{proof}
(i)
The mCH equation \eqref{mCH} can be rewritten as
\begin{align}\label{burgers1}
m_t+(u^2-u_x^2)m_x=-2m^2u_x.
\end{align}
Therefore, we have
\begin{align*}
\frac{d}{dt}m(X(\xi,t),t)=-2(m^2u_x)(X(\xi,t),t).
\end{align*}
By \eqref{keepsign}, when $m_0(\xi_0)\neq0$ we know $m(X(\xi,t),t)\neq0$ and it will keep sign (positive or negative) for $t\in[0,T_{max})$. Hence
\begin{align}\label{minequa1}
\frac{1}{m^2(X(\xi_0,t),t)}\frac{d}{dt}m(X(\xi_0,t),t)=-2u_x(X(\xi_0,t),t).
\end{align}
This implies
\begin{align*}
\frac{d}{dt}\bigg(\frac{1}{m(X(\xi_0,t),t)}\bigg)=2u_x(X(\xi_0,t),t).
\end{align*}
Integrating from $0$ to $t$ leads to
\begin{align}\label{burgers}
\frac{1}{m(X(\xi_0,t),t)}=\int_0^t 2u_x(X(\xi_0,s),s)ds+\frac{1}{m_0(\xi_0)},
\end{align}
and combining \eqref{keepsign} gives \eqref{blowupintegration}.\\
(ii)
If $T_{max}<\frac{1}{||m_0||_{L^\infty}||m_0||_{L^1}}$, then \eqref{blowupintegration} and \eqref{uuxUbound} give that
$$X_\xi(\xi_0,T_{max})\geq 1-||m_0||_{L^\infty}||m_0||_{L^1}T_{max}>0,$$
which is a contradiction with the assumption of blow-up at $T_{max}$.

\end{proof}

In view of  equation \eqref{burgers1}, the most natural way  to study blow-up behavior is following the characteristics. This method was used for the Burgers equation and the CH equation.
Equality \eqref{burgers} reminds us the proof for finite time blow-up of Burgers equation:
\begin{align}\label{burgersequation}
u_t+uu_x=0, ~\textrm{ for }~x\in\mathbb{R},~t>0.
\end{align}
Consider its characteries $\dot{X}(x,t)=u(X(x,t),t)$ and we have
$$\frac{d}{dt}u(X(x,t),t)=0.$$
Taking derivative of \eqref{burgersequation} gives
$$u_{xt}+u_x^2+uu_{xx}=0.$$
Then we have
$$\frac{d}{dt}u_x(X(x,t),t)=(uu_{xx})(X(x,t),t)+u_{xt}(X(x,t),t)=-u_x^2(X(x,t),t),$$
which implies
\begin{align}\label{63}
\frac{1}{u_x(X(x,t),t)}=t+\frac{1}{u_{0x}(x)}.
\end{align}
Hence, if there exists $x_0\in\mathbb{R}$ such that $u_{0x}(x_0)<0$, then $u_x$ goes to $-\infty$ in finite time.

\eqref{burgers} is similar to \eqref{63}. But we can not have direct estimate on the blow-up time like the Burgers equation. Hence we need to give some estimate about $u_x$. We have the following lemma.

\begin{lemma}\label{preparelemma2}
Suppose $m_0\in C_c^1(-L,L)$ and $M_1:=||m_0||_{L^1}$. Let $T_{max}$ be the maximal time of existence of the corresponding classical solution $m(x,t)$ to \eqref{mCH}-\eqref{initial m}.
 $X\in C_1^2([-L,L]\times[0,T_{max}))$ is the solution to \eqref{Lagran dynamics}. Then we have
\begin{align}\label{dynamicsfor ux}
\bigg|\frac{d}{dt}u_x(X(\xi,t),t)\bigg|\leq \frac{M_1^3}{2}.
\end{align}

\end{lemma}
\begin{proof}

From \eqref{mCH}, we obtain
\begin{align*}
(1-\partial_{xx})u_t+(1-\partial_{xx})[(u^2-u_x^2)u_x]&=-[(u^2-u_x^2)m]_x+(1-\partial_{xx})[(u^2-u_x^2)u_x]\\
&=-2u_xm^2-6u_xu_{xx}m-2u_x^2m_x,
\end{align*}
which implies
\begin{align}\label{minquality2}
u_t+(u^2-u_x^2)u_x&=-(1-\partial_{xx})^{-1}[2u_xm^2+6u_xu_{xx}m+2u_x^2m_x].
\end{align}
Taking derivative to \eqref{minquality2} with respect to $x$ yields
\begin{align*}
u_{xt}+2u_x^2m+(u^2-u_x^2)u_{xx}=-\partial_x(1-\partial_{xx})^{-1}[2u_xm^2+6u_xu_{xx}m+2u_x^2m_x]\\
=-\partial_x(1-\partial_{xx})^{-1}[2u^2u_x+2uu_xu_{xx}-4u_xu_{xx}^2+2u_x^3-2u_x^2u_{xxx}].
\end{align*}
Due to
$\partial_{xx}(u_x^3)=6u_xu_{xx}^2+3u_x^2u_{xxx}$ and $\partial_{xx}(uu_x^2)=5u_x^2u_{xx}+2uu_{xx}^2+2u_xu_{xxx}$,
we have
\begin{align*}
&\quad u_{xt}+2u_x^2m+(u^2-u_x^2)u_{xx}\\
&=-\partial_x(1-\partial_{xx})^{-1}\Big[2u^2u_x+2uu_xu_{xx}+2u_x^3+\frac{2}{3}(1-\partial_{xx})u_x^3-\frac{2}{3}u_x^3\Big]\\
&=-2u_x^2u_{xx}-(1-\partial_{xx})^{-1}\Big[5uu_x^2+2u^2u_{xx}+u_x^2u_{xx}-(1-\partial_{xx})(uu_x^2)\Big]\\
&=uu_x^2-2u_x^2u_{xx}-(1-\partial_{xx})^{-1}\Big[5uu_x^2+2u^2u_{xx}+u_x^2u_{xx}\Big].
\end{align*}
After some calculation we obtain
\begin{align*}
u_{xt}+(u^2-u_x^2)u_{xx}=-uu_x^2-(1-\partial_{xx})^{-1}[5uu_x^2+2u^2u_{xx}+u_x^2u_{xx}].
\end{align*}
For the last two terms on the right side, integration by parts shows that
\begin{align*}
-(1-\partial_{xx})^{-1}[2u^2u_{xx}+u_x^2u_{xx}]=-G\ast\Big(2u(uu_x)_x-2uu_x^2+\frac{1}{3}(u_x^3)_x\Big)\\
=-G\ast\Big(2(u^2u_x)_x-4uu_x^2+\frac{1}{3}(u_x^3)_x\Big)=4G\ast(uu_x^2)-G'\ast\Big(2u^2u_x+\frac{1}{3}u_x^3\Big).
\end{align*}
Hence
\begin{align*}
u_{xt}+(u^2-u_x^2)u_{xx}&=-uu_x^2-G\ast(uu_x^2)-G'\ast\Big(\frac{2}{3}(u^3)_x+\frac{1}{3}u_x^3\Big)\\
&=-uu_x^2-G\ast(uu_x^2)+\frac{2}{3}u^3-\frac{2}{3}G\ast(u^3)-G'\ast\Big(\frac{1}{3}u_x^3\Big)
\end{align*}
Young's inequality and  \eqref{uuxUbound}  give that
\begin{align*}
|u_{xt}+(u^2-u_x^2)u_{xx}|&\leq ||uu_x^2||_{L^\infty}+\Big|\Big|uu_x^2-\frac{2}{3}u^3\Big|\Big|_{L^\infty}||G||_{L^1}+\frac{2}{3}||u||_{L^\infty}^3+\Big|\Big|\frac{1}{3}u_x^3\Big|\Big|_{L^\infty}||G'||_{L^1}\\
&\leq\frac{1}{2}M_1^3,
\end{align*}
which implies \eqref{dynamicsfor ux}.

\end{proof}

Next, we state and prove our main results in this section.

\begin{theorem}\label{finite time blow up theorem}
Suppose $m_0\in C_c^1(-L,L)$ and $M_1:=||m_0||_{L^1}$.  Let $T_{max}$ be the maximal time of existence of the classical solution $m(x,t)$ to \eqref{mCH}-\eqref{initial m}. $X\in C_1^2([-L,L]\times[0,T_{max}))$ is the solution to \eqref{Lagran dynamics}. If there is a $\xi_0\in[-L,L]$ such that $m_0(\xi_0)>0$ and
\begin{align}\label{initialcondition}
-\partial_xu_0(\xi_0)>\sqrt{\frac{M_1^3}{2m_0(\xi_0)}},
\end{align}
then $m(x,t)$ defined by \eqref{eq:umdef} blows up at a time
\begin{equation}\label{blowuptimelessthen}
T_{max}\leq t^*:=\frac{2}{M_1^3}\bigg(-\partial_xu_0(\xi_0)-\sqrt{\Big[\partial_xu_0(\xi_0)\Big]^2-\frac{M_1^3}{2m_0(\xi_0)}}\bigg).
\end{equation}

Moreover, when $T_{max}=t^*$, we have the following estimate of the blow-up rate for $m$:
\begin{equation}\label{blow-up rate for m)}
||m(\cdot ,t)||_{L^\infty}\geq\frac{1}{C(T_{max}-t)}\textrm{ for }t\in[0,T_{max}),
\end{equation}
and for $X_\xi$ we have
\begin{equation}\label{blow-up rate for lagrange}
\inf_{\xi\in(-L,L)}X_\xi(\xi,t)\leq Cm_0(\xi_0)(T_{max}-t)\textrm{ for }t\in[0,T_{max}),
\end{equation}
Where
$$C=-\partial_xu_0(\xi_0)+\sqrt{\Big[\partial_xu_0(\xi_0)\Big]^2-\frac{M_1^3}{2m_0(\xi_0)}}.$$

\end{theorem}

\begin{proof}
\emph{Step 1.}

Assume $m_0(\xi_0)>0$.
Combining \eqref{minequa1} and  \eqref{dynamicsfor ux} shows that
\begin{align}\label{minequa2}
\frac{d}{dt}\bigg(\frac{1}{m^2(X(\xi_0,t),t)}\frac{d}{dt}m(X(\xi_0,t),t)\bigg)= -2 \frac{d}{dt}u_x(X(\xi_0,t),t)\geq -M_1^3.
\end{align}
Integrating \eqref{minequa2}  shows that
\begin{align}\label{minequa3}
\frac{1}{m^2(X(\xi_0,t),t)}\frac{d}{dt}m(X(\xi_0,t),t)\geq -M_1^3t-2\partial_xu_0(\xi_0)
\end{align}
where we used
$$\frac{1}{m^2(X(\xi_0,t),t)}\frac{d}{dt}m(X(\xi_0,t),t)\bigg|_{t=0}=-2\partial_xu_0(\xi_0).$$
Integrating \eqref{minequa3} gives
\begin{align*}
\frac{1}{m(X(\xi_0,t),t)}\leq\frac{1}{2}M_1^3t^2+2\partial_xu_0(\xi_0)t+\frac{1}{m_0(\xi_0)}.
\end{align*}
If $\xi_0$ satisfies \eqref{initialcondition}, then we have
$$\frac{1}{2}M_1^3t^2+2\partial_xu_0(\xi_0)t+\frac{1}{m_0(\xi_0)}=\frac{1}{2}M_1^3(t-t^*)(t-t_*),$$
where
$$t^*=\frac{2}{M_1^3}\bigg(-\partial_xu_0(\xi_0)-\sqrt{\Big[\partial_xu_0(\xi_0)\Big]^2-\frac{M_1^3}{2m_0(\xi_0)}}\bigg)$$
and
$$t_*=\frac{2}{M_1^3}\bigg(-\partial_xu_0(\xi_0)+\sqrt{\Big[\partial_xu_0(\xi_0)\Big]^2-\frac{M_1^3}{2m_0(\xi_0)}}\bigg).$$
Hence
\begin{align}\label{minequa5}
0<\frac{1}{m(X(\xi_0,t),t)}\leq\frac{1}{2}M_1^3(t-t^*)(t-t_*).
\end{align}
This implies that there is a time $0<T_{max}\leq t^*$ such that
$$m(X(\xi_0,t),t)\rightarrow+\infty,\textrm{ as }t\rightarrow T_{max}$$
which means $m(x,t)$ blows up at the time $T_{max}.$

\emph{Step 2.}

Assume $T_{max}=t^*.$ From \eqref{minequa5}, we have
\begin{align*}
||m(\cdot ,t)||_{L^\infty}&\geq m(X(\xi_0,t),t)\geq\frac{2}{M_1^3(t-t^*)(t-t_*)}\\
&\geq\frac{2}{M_1^3t_*(T_{max}-t)}=\frac{1}{C(T_{max}-t)}.
\end{align*}
Hence, we have \eqref{blow-up rate for m)}.

From \eqref{keepsign} and \eqref{minequa5}, we have
\begin{align*}
\inf_{\xi\in(-L,L)}X_\xi(\xi,t)&\leq X_\xi(\xi_0,t)=\frac{m_0(\xi_0)}{m(X(\xi_0,t),t)}\leq\frac{1}{2}m_0(\xi_0)M_1^3(t-t^*)(t-t_*)\\
&\leq \frac{1}{2}m_0(\xi_0)M_1^3t_*(T_{max}-t)\leq Cm_0(\xi_0)(T_{max}-t).
\end{align*}
Hence, \eqref{blow-up rate for lagrange} follows and this ends the proof.

\end{proof}

Similarly, we have the following theorem.

\begin{theorem}\label{finite time blow up theorem 2}
Suppose $m_0\in C_c^1(-L,L)$ and $M_1:=||m_0||_{L^1}$.  Let $T_{max}$ be the maximal time of existence of the classical solution $m(x,t)$ to \eqref{mCH}-\eqref{initial m}. $X\in C_1^2([-L,L]\times[0,T_{max}))$ is the solution to \eqref{Lagran dynamics}. If there is a $\xi_1\in[-L,L]$ such that $m_0(\xi_1)<0$ and
\begin{align}\label{initialcondition2}
\partial_xu_0(\xi_1)>\sqrt{\frac{M_1^3}{-2m_0(\xi_1)}},
\end{align}
then $m(x,t)$ defined by \eqref{eq:umdef} blows up at a time
\begin{equation*}
T_{max}\leq t^*:=\frac{2}{M_1^3}\bigg(\partial_xu_0(\xi_1)-\sqrt{\Big[\partial_xu_0(\xi_1)\Big]^2+\frac{M_1^3}{2m_0(\xi_1)}}\bigg).
\end{equation*}

Moreover, when $T_{max}=t^*$, we have the following estimate of the blow-up rate for $m(x,t)$:
\begin{equation*}
||m(\cdot ,t)||_{L^\infty}\geq\frac{1}{C(T_{max}-t)}\textrm{ for }t\in[0,T_{max}),
\end{equation*}
and for $X_\xi$ we have
\begin{equation*}
\inf_{\xi\in(-L,L)}X_\xi(\xi,t)\leq Cm_0(\xi_1)(t-T_{max})\textrm{ for }t\in[0,T_{max}),
\end{equation*}
Where
$$C=\partial_xu_0(\xi_1)+\sqrt{\Big[\partial_xu_0(\xi_1)\Big]^2+\frac{M_1^3}{2m_0(\xi_1)}}.$$

\end{theorem}

From conditions \eqref{initialcondition} and \eqref{initialcondition2},  if there exists  $\overline{\xi}\in[-L,L]$ such that \eqref{initialu0} holds,
then the classical solution will blow up in finite time.

Now we can prove Theorem \ref{lifespan}.
\begin{proof}[Proof of Theorem \ref{lifespan}]
(i) \eqref{lifespanlowbound} follows from \eqref{blowupbelowbound}.

(ii) Let $m_0$ satisfies the assumptions in Theorem \ref{finite time blow up theorem}. Then, for any $\epsilon>0$ we know $\epsilon m_0$ also satisfies the assumptions. Hence, from  \eqref{blowuptimelessthen} we have
$$T_{max}(\epsilon m_0)\leq \frac{2||\epsilon u_x||_{L^\infty}}{||\epsilon m_0||^3_{L^1}}\leq \frac{1}{||m_0||^2_{L^1}}\cdot\frac{1}{\epsilon^2},$$
where \eqref{uuxUbound} was used.
Together with \eqref{lifespanlowbound} we can obtain \eqref{lifespanupbound}.
\end{proof}

\section{Solutions at blow-up time  and formation of peakons}\label{sec4}
In this section, we study the behavior of  classical solutions at  blow-up time $T_{max}$.

First, we show that $u$ and $u_x$ are uniformly BV function for $t\in[0,T_{max}]$ (including the blow-up time $T_{max}$) and $m(\cdot ,t)$ has a unique limit in Radon measure space as  $t$ approaching  $T_{max}$.

Let us recall the concept of the space $BV(\mathbb{R})$.
\begin{definition}\label{BV}
$\mathrm{(i)}$
For dimention $d\geq1$ and an open set $\Omega\subset\mathbb{R}^d$, a function $f\in L^1(\Omega)$ belongs to $BV(\Omega)$ if
$$Tot.Var.\{f\}:=\sup\Big\{\int_{\Omega}f(x)\nabla\cdot\phi(x)dx:\phi\in C_c^1(\Omega;\mathbb{R}^d), ||\phi||_{L^\infty}\leq1\Big\}<\infty.$$
$\mathrm{(ii)}$ (Equivalent definition for one dimension case) A function $f$ belongs to $BV(\mathbb{R})$ if for any $\{x_i\}\subset\mathbb{R}$, $x_i<x_{i+1}$, the following statement holds:
$$Tot.Var.\{f\}:=\sup_{\{x_i\}}\Big\{\sum_i|f(x_i)-f(x_{i-1})|\Big\}<\infty.$$
\end{definition}
\begin{remark}
Let $\Omega\subset\mathbb{R}^d$ for $d\geq1$ and $f\in BV(\Omega)$. $Df:=(D_{x_1}f,\ldots,D_{x_d}f)$ is the distributional gradient of $f$. Then, $Df$ is a vector Radon measure and  the total variation of $f$ is equal to the total variation of $|D f|$: $Tot.Var.\{f\}=|D f|(\Omega).$  Here, $|D f|$ is the total variation measure of the vector measure $Df$ (\cite[Definition (13.2)]{Leoni}).

If a function $f:\mathbb{R}\rightarrow \mathbb{R}$ satisfies Definition \ref{BV} (ii), then $f$ satisfies Definition (i). On the contrary, if $f$ satisfies Definition \ref{BV} (i), then there exists a right continuous representative which satisfies Definition (ii). See \cite[Theorem 7.2]{Leoni} for the proof.
\end{remark}

We have the following theorem about $u$ and $u_x$ at $T_{max}$.
\begin{theorem}\label{uniform bv}
Let $m_0\in C_c^1(-L,L)$ and $M_1:=||m_0||_{L^1}$. Let $T_{max}$ be the maximal existence time for the classical solution $m(x,t)$ to \eqref{mCH}-\eqref{initial m} and $X\in C^{2}_1([-L,L]\times[0,T_{max}])$ be the solution to \eqref{Lagran dynamics}. Then, the following assertions hold:\\
$(\mathrm{i})$
There exists a function $u(x,T_{max})$ such that
\begin{align}\label{BV1}
\lim_{t\rightarrow T_{max}}u(x,t)=u(x,T_{max}),\quad \lim_{t\rightarrow T_{max}}u_x(x,t)=u_x(x,T_{max})\textrm{ for every }x\in\mathbb{R}.
\end{align}
$(\mathrm{ii})$
For any $t\in[0,T_{max}]$ we have
$$u(\cdot,t),~~u_x(\cdot,t)\in BV(\mathbb{R})$$
and
\begin{align}\label{uniforBV}
\mathrm{Tot.Var.}\{u(\cdot, t)\}\leq M_1,\quad \mathrm{Tot.Var.}\{u_x(\cdot, t)\}\leq 2M_1.
\end{align}

\end{theorem}
\begin{proof}

We use three steps to prove (i) and (ii) together.

\emph{Step 1.} We prove $u\in C(\mathbb{R}\times[0,T_{max}])$.

Due to \eqref{Xtmax}  and $u(x,t)=\int_{-L}^LG(x-X(\theta,t))m_0(\theta)d\theta$ for $t\in[0,T_{max})$, let $t$ go to $T_{max}$  and we obtain
$$u(x,T_{max})=\int_{-L}^LG(x-X(\theta,T_{max}))m_0(\theta)d\theta.$$
Moreover, we have $u\in C(\mathbb{R}\times[0,T_{max}])$.

\emph{Step 2.} For $0\leq t<T_{max}$, we prove \eqref{uniforBV}.

For $G=\frac{1}{2}e^{-|x|}$, we know $G,~G_x\in BV(\mathbb{R})$ and the following holds
$$\mathrm{Tot.Var.}\{G\}=1,\quad \mathrm{Tot.Var.}\{G_x\}=2.$$
 When $t\in[0,T_{max})$, for any $\{x_i\}\subset\mathbb{R}$, $x_i<x_{i+1}$, we have
\begin{align*}
\sum_i|u(x_{i},t)-u(x_{i-1},t)|&\leq\int_{-L}^L\sum_i|G(x_i-X(\theta,t))-G(x_{i-1}-X(\theta,t))||m_0(\theta)|d\theta\\
&\leq \mathrm{Tot.Var.}\{G\}||m_0||_{L^1}=M_1,
\end{align*}
which means $\mathrm{Tot.Var.}\{u(\cdot, t)\}\leq M_1.$  Similarly, we can obtain $\mathrm{Tot.Var.}\{u_x(\cdot, t)\}\leq 2M_1$ for $t\in[0,T_{max}).$

\emph{Step 3.} We prove \eqref{BV1} and show that $u(x,T_{max})$ satisfies \eqref{uniforBV}.

The first part of \eqref{BV1} is deduced by $u\in C(\mathbb{R}\times[0,T_{max}])$. To prove the second part, we have to do a little more job.

Combining  \eqref{uuxUbound}, step 2, and \cite[Theorem 2.3]{Bressan}, we know that there exists a consequence $\{t_k\}(\rightarrow T_{max})$ and two BV functions $\widetilde{u}(x),\widetilde{v}(x)$ such that
$$\lim_{k\rightarrow\infty}u(x,t_k)=\widetilde{u}(x),\quad\lim_{k\rightarrow\infty}u_x(x,t_k)=\widetilde{v}(x)\textrm{ for every }x\in\mathbb{R},$$
and
$$\mathrm{Tot.Var.}\{\widetilde{u}\}\leq M_1,\quad |\widetilde{u}|\leq \frac{1}{2}M_1 \textrm{ and } \mathrm{Tot.Var.}\{\widetilde{v}\}\leq 2M_1,\quad |\widetilde{v}|\leq \frac{1}{2}M_1.$$
Because
$$\lim_{t\rightarrow T_{max}}u(x,t)=u(x,T_{max})\textrm{ for every }x\in\mathbb{R},$$
we know $\widetilde{u}(x)=u(x,T_{max}).$

For any test function $\phi\in C_c^\infty(\mathbb{R})$, we have
\begin{align*}
-\int_{\mathbb{R}}u(x,T_{max})\phi_x(x)dx&=-\int_{\mathbb{R}}\widetilde{u}(x)\phi_x(x)dx=-\lim_{k\rightarrow\infty}\int_{\mathbb{R}}u(x,t_k)\phi_x(x)dx\\
&=\lim_{k\rightarrow\infty}\int_{\mathbb{R}}u_x(x,t_k)\phi(x)dx=\int_{\mathbb{R}}\widetilde{v}(x)\phi(x)dx,
\end{align*}
which means $\widetilde{v}(x)$ is the derivative of $u(x,T_{max})$ in distribution sense. Define $u_x(x,T_{max})=\widetilde{v}(x)$ for every $x\in\mathbb{R}$ and we obtain
$$\lim_{k\rightarrow\infty}u_x(x,t_k)=\widetilde{u}_x(x)=u_x(x,T_{max})\textrm{ for every }x\in\mathbb{R}.$$
Because $u_x(x,t)$ is continuous in $[0,T_{max})$,  we know  
$$\lim_{t\rightarrow T_{max}}u_x(x,t)=\widetilde{u}_x(x)=u_x(x,T_{max})\textrm{ for every }x\in\mathbb{R}.$$

This is the end of the proof.

\end{proof}

Next we give a theorem to prove that $m(\cdot ,t)$ has a unique limit in Radon measure space $\mathcal{M}(\mathbb{R})$ as $t$ approaching $T_{max}$. Before this, let's recall the definition $A^+_t$ and $A^-_t$ in Remark \ref{positivenegative} and denote
$$A_{T_{max}}^+:=\{X(\xi,T_{max})\in\mathbb{R}:\xi\in A^+\},\quad A_{T_{max}}^-:=\{X(\xi,T_{max})\in\mathbb{R}:\xi\in A^-\}.$$
Because $X(\xi,T_{max})$ may not be strictly monotonic, it is not obvious to see that $A_{T_{max}}^+$ and $A_{T_{max}}^-$ are open sets. We give a lemma to show this.
\begin{lemma}\label{openset}
$A_{T_{max}}^+$ and $A_{T_{max}}^-$ are open sets.
\end{lemma}
\begin{proof}
We only deals with $A_{T_{max}}^+$ and the proof for $A_{T_{max}}^-$ is similar.

For $x_0\in A_{T_{max}}^+$, there exist $\xi\in(-L,L)$ such that $m_0(\xi)>0$ and $x_0=X(\xi,T_{max}).$ Set
$$\xi_1:=\min\{\xi\in[-L,L]:m_0(\xi)\geq0~\textrm{ and }~X(\xi,T_{max})=x_0\}$$
and
$$\xi_2:=\max\{\xi\in[-L,L]:m_0(\xi)\geq0~\textrm{ and }~X(\xi,T_{max})=x_0\}.$$
By continuity of $m_0$ and $X(\xi,T_{max})$, $\xi_1$ and $\xi_2$ can be obtained.

1.
If $\xi_1=\xi_2$, then there is only one point $\xi_0=\xi_1$ such that $m_0(\xi_0)>0$ and $x_0=X(\xi_0,T_{max}).$ In this case, set
$$\eta_1:=\max\{\xi:m_0(\xi)=0~\textrm{ and }~\xi<\xi_0\}$$
and
$$\eta_2:=\min\{\xi:m_0(\xi)=0~\textrm{ and }~\xi>\xi_0\}.$$
Because $m_0(\xi_0)>0$, we know $\eta_1<\xi_0<\eta_2$ and $m_0(\xi)>0$ for $\xi\in(\eta_1,\eta_2)$. Hence
$$X(\xi,T_{max})\in A_{T_{max}}^+,~\textrm{ for }~\xi\in(\eta_1,\eta_2).$$
Because $X(\xi,T_{max})$ is nondecreasing, we obtain
$$x_0=X(\xi_0,T_{max})\in(X(\eta_1,T_{max}),X(\eta_2,T_{max}))\subset A_{T_{max}}^+.$$

2.
If $\xi_1<\xi_2$, we have
\begin{align}\label{xi1xi2}
X(\xi,T_{max})\equiv x_0,~\textrm{ for }~\xi\in[\xi_1,\xi_2].
\end{align}
By definition we know $m_0(\xi_i)\geq0$ for $i=1,2$. When $m_0(\xi_i)=0$ for $i=1$ or $i=2$, from Remark \ref{positivenegative} we know $X_\xi(\xi_i, T_{max})=1$. This implies that $X(\xi,T_{max})$ is strictly monotonic in a neighborhood of $\xi_i$ which is a contradiction with \eqref{xi1xi2}. Hence, we have
$$m_0(\xi_i)>0,~\textrm{ for }~i=1,2.$$
Hence, there exist $\xi_3<\xi_1$ and $\xi_4>\xi_2$ such that
$$m_0(\xi)>0~\textrm{ for }~\xi\in(\xi_3,\xi_4).$$
Therefore, we have
$$X(\xi_3,T_{max})<X(\xi_1,T_{max})=x_0=X(\xi_2,T_{max})<X(\xi_4,T_{max}),$$
and
$$X(\xi,T_{max})\in A_{T_{max}}^+~\textrm{ for }~\xi\in (\xi_3,\xi_4),$$
which imply
$$x_0=X(\xi_1,T_{max})\in(X(\xi_3,T_{max}),X(\xi_4,T_{max}))\subset A_{T_{max}}^+.$$
\end{proof}

For any Radon measure $\mu$ and measurable set $A$, we use $\mu|_A$ to stand for the restriction of $\mu$ on the set $A$. We have the following Theorem.
\begin{theorem}\label{RadonmeasureatT}
Let the assumptions in Theorem \ref{uniform bv} holds.
Then there exists a unique Radon measure $m(\cdot ,T_{max})$ such that
\begin{align}\label{convergenceradonmeasure}
m(\cdot ,t)\stackrel{\ast}{\rightharpoonup} m(\cdot ,T_{max})~\textrm{ in }~\mathcal{M}(\mathbb{R}), ~\textrm{ as }~t\rightarrow T_{max}.
\end{align}
Moreover, $m(\cdot ,T_{max})$ has the following properties:\\
$\mathrm{(i)}$
Compact support:
\begin{equation}\label{supportof mTmax}
\mathrm{supp}\{m(\cdot ,T_{max})\}\subset(-L,L).
\end{equation}
$\mathrm{(ii)}$
Denote
$$m^+_{T_{max}}:=m(\cdot ,T_{max})\Big|_{A_{T_{max}}^+}~\textrm{ and }~m^-_{T_{max}}:=m(\cdot ,T_{max})\Big|_{A_{T_{max}}^-}.$$
Then $m^+_{T_{max}}$ is a positive Radon measure and $m^-_{T_{max}}$ is a negative Radon measure. Besides, we have
\begin{align}\label{decomposition of m}
m(\cdot ,T_{max})=m^+_{T_{max}}+m^-_{T_{max}}.
\end{align}
$\mathrm{(iii)}$
 The following equality holds:
\begin{align}\label{keepmass}
\int_{\mathbb{R}}|m|(dx,T_{max})=\int_{\mathbb{R}}|m(x,t)|dx=\int_{-L}^L|m_0(x)|dx,~~t\in[0,T_{max}).
\end{align}

\end{theorem}

\begin{proof}
\emph{Step 1.} Proof of \eqref{convergenceradonmeasure}.

Because $u_x(\cdot ,T_{max})$ is a BV function,  its derivative $u_{xx}(\cdot ,T_{max})$ is a Radon measure. We know
$$m(\cdot ,T_{max})=u(\cdot ,T_{max})-u_{xx}(\cdot ,T_{max})$$
is a Radon measure and for any test function $\phi\in C_c^\infty(\mathbb{R})$,  we have
\begin{align}\label{radon}
\int_{\mathbb{R}}\phi(x)m(dx,T_{max})=\int_{\mathbb{R}}u(x,T_{max})\phi(x)+u_x(x,T_{max})\phi_x(x)dx.
\end{align}
Then,  we have
\begin{align*}
\lim_{t\rightarrow T_{max}}\int_{\mathbb{R}}m(x,t)\phi(x)dx&=\lim_{t\rightarrow T_{max}}\int_{\mathbb{R}}u(x,t)\phi(x)+u_x(x,t)\phi_x(x)dx\\
&=\int_{\mathbb{R}}u(x,T_{max})\phi(x)+u_x(x,T_{max})\phi_x(x)dx\\
&=\int_{\mathbb{R}}\phi(x)m(dx,T_{max}).
\end{align*}
This proves \eqref{convergenceradonmeasure}.

\emph{Step 2.} Proof of (i).

For any test function $\phi\in C_c^\infty(\mathbb{R})$,  we have
\begin{align}\label{testfunctionforradon}
\int_{\mathbb{R}}\phi(x)m(dx,T_{max})&=\lim_{t\rightarrow T_{max}}\int_{\mathbb{R}}m(x,t)\phi(x)dx\nonumber\\
&=\lim_{t\rightarrow T_{max}}\int_{\mathbb{R}}m(X(\xi,t),t)\phi(X(\xi,t))X_\xi(\xi,t)d\xi\nonumber\\
&=\lim_{t\rightarrow T_{max}}\int_{-L}^Lm_0(\xi)\phi(X(\xi,t))d\xi\nonumber\\
&=\int_{-L}^Lm_0(\xi)\phi(X(\xi,T_{max}))d\xi,
\end{align}
where \eqref{keepsign} was used. Because $X(L,t)=L$ and $X(-L,t)=-L$ for $t\in[0,T_{max}]$, we have $X(\cdot ,T_{max})\in(-L,L)$. Let test function $\phi$ satisfy $\mathrm{supp}\{\phi\}\subset \mathbb{R}\setminus(-L,L).$
Then we obtain
$$\int_{\mathbb{R}}\phi(x)m(dx,T_{max})=\int_{-L}^Lm_0(\xi)\phi(X(\xi,T_{max}))d\xi=0,$$
which implies \eqref{supportof mTmax}.

\emph{Step 3.} Proof of (ii).

Due to \eqref{testfunctionforradon}, we know 
$$m(\cdot,T_{max})=X(\cdot,T_{max})\#m_0.$$
For $\phi\in C_c^\infty(\mathbb{R})$ and $\phi\geq0$, by the definition of $A^+$  we have
\begin{align*}
\int_{\mathbb{R}}\phi(x)dm_{T_{max}}^+&=\int_{A_{T_{max}}^+}\phi(x)m(dx,T_{max})\\
&=\int_{A^+}m_0(\xi)\phi(X(\xi,T_{max}))d\xi\geq0.
\end{align*}
Hence, $m_{T_{max}}^+$ is a positive Radon measure. With the same argument, we can see that $m_{T_{max}}^-$ is a negative Radon measure.

On the other hand, by using \eqref{testfunctionforradon}, we have
\begin{align*}
&\quad \int_{\mathbb{R}}\phi(x)d(m_{T_{max}}^+ + m_{T_{max}}^-)=\int_{A_{T_{max}}^+\cup A_{T_{max}}^-}\phi(x)m(dx,T_{max})\\
&=\int_{A^+\cup A^-}m_0(\xi)\phi(X(\xi,T_{max}))d\xi=\int_{A^+\cup A^-\cup A^0}m_0(\xi)\phi(X(\xi,T_{max}))d\xi\\
&=\int_{-L}^Lm_0(\xi)\phi(X(\xi,T_{max}))d\xi=\int_{\mathbb{R}}\phi(x)m(dx,T_{max}),
\end{align*}
which implies \eqref{decomposition of m}.

\emph{Step 4.} Proof of (iii).

 From \eqref{keepsign}, we have
$$|m(X(\xi,t),t)|X_\xi(\xi,t)=|m_0(\xi)|,$$
which implies
\begin{align*}
\int_{\mathbb{R}}|m(x,t)|dx=\int_{\mathbb{R}}|m(X(\xi,t),t)|X_\xi(\xi,t)d\xi=\int_{-L}^L|m_0(\xi)|d\xi~\textrm{ for }~t\in[0,T_{max}).
\end{align*}
For any test function $\phi\in C_c^\infty(\mathbb{R})$, we have
\begin{align*}
\int_{\mathbb{R}}\phi(x)|m|(dx,T_{max})&=\int_{A_{T_{max}}^+}\phi(x) m(dx,T_{max})-\int_{A_{T_{max}}^-}\phi(x)m(dx,T_{max})\\
&=\int_{A^+}m_0(\xi)\phi(X(\xi,T_{max}))d\xi-\int_{A^-}m_0(\xi)\phi(X(\xi,T_{max}))d\xi.
\end{align*}
Choose $\phi\in C_c^\infty(\mathbb{R})$ satisfying
$$\phi(x)\equiv1,\quad x\in(X(-L,T_{max}),X(L,T_{max})).$$
Hence, we have
\begin{align*}
\int_{\mathbb{R}}|m|(dx,T_{max})=\int_{A^+}m_0(\xi)d\xi-\int_{A^-}m_0(\xi)d\xi=\int_{-L}^L|m_0(\xi)|d\xi.
\end{align*}

This ends the proof.
\end{proof}

\begin{remark}\label{solutionsforrandomeasure}
In Section \ref{sec5}, we will prove the global existence of weak solutions to the mCH equation when initial data $m_0$ belongs to $\mathcal{M}(\mathbb{R})$.  Hence, we can extend $m$ globally in time after blow up time. Similar results can be found  in \cite{GaoLiu}, where a sticky particle  method was used.

\end{remark}

 Next, we introduce another two sets to study solutions at $T_{max}$.
Assume $m_0\in C^1_c(\mathbb{R})$ and $X\in C^{2}_1([-L,L]\times[0,T_{max}])$ is the solution to the Lagrange dynamics \eqref{Lagran dynamics}.
Set
$$F:=\{\xi\in(-L,L): X_\xi(\xi,T_{max})=0\}$$ and
$$O:=\{\xi\in(-L,L): X_\xi(\xi,T_{max})>0\}.$$
Then, $F$ is a closed set and $O$ is an open set. Moreover, we have
$$F\cup O=(-L,L).$$
Because the classical solution blows up in finite time $T_{max}$, we know $F$ is not empty. On the other hand, due to $m_0(\pm L)=0$, Remark \ref{positivenegative} tells that $X_\xi(\pm L, T_{max})=1$ which implies $O$ is not empty.

Set
\begin{align}\label{FTmax}
O_{T_{max}}:=\{X(\xi,T_{max}):\xi\in O\}~\textrm{ and }~F_{T_{max}}:=\{X(\xi,T_{max}):\xi\in F\}.
\end{align}
Then, we have
$$O_{T_{max}}\cup F_{T_{max}}=(X(-L,T_{max}),X(L,T_{max})).$$
$X(\cdot ,T_{max})$ is strictly monotonic in $O$. Hence, $O_{T_{max}}$ is also an open set and $F_{T_{max}}$ is a closed set. Moreover, we claim that
\begin{align}\label{eq:dense}
\overline{O_{T_{max}}}=[X(-L,T_{max}),X(L,T_{max})].
\end{align}
To show \eqref{eq:dense}, we only have to prove $F_{T_{max}}\subset\overline{O_{T_{max}}}$. For any $x\in F_{T_{max}}$, there exists $\xi_0$ such that $x=X(\xi_0,T_{max})$ and $X_\xi(\xi_0,T_{max})=0$. Let $\xi_1=\max\{\xi:X(\xi,T_{max})=x\}.$
Then there is a small constant $\delta$ such that $\xi_1+\delta<L$ and $X_\xi(\xi,T_{max})>0$ for $\xi\in(\xi_1,\xi_1+\delta)$. Hence, $X(\xi,T_{max})\in O_{T_{max}}$ for $\xi\in(\xi_1,\xi_1+\delta)$ and $\lim_{\xi\to\xi_1+}=X(\xi_1,T_{max})=x,$ which implies $x\in \overline{O_{T_{max}}}$.

%
%

We have the following theorem.

\begin{theorem}\label{continuousinotmax}
Let assumptions in Theorem \ref{uniform bv} hold. Then we have
$$u(\cdot ,T_{max})\in C^{3}(\mathbb{R}\setminus F_{T_{max}})$$
and
$$m(\cdot ,T_{max})\in C^{1}(O_{T_{max}})\cap L^1(O_{T_{max}}).$$
Moreover, the following holds
\begin{align*}
m(X(\xi,T_{max}),T_{max})X_\xi(\xi,T_{max})=m_0(\xi)~\textrm{ for }~\xi\in O.
\end{align*}
\end{theorem}
\begin{proof}
\emph{Step 1.} We first consider the cases when $x\not\in(X(-L,T_{max}),X(L,T_{max})).$

Because $m_0(L)=0$, from Remark \ref{positivenegative} we know $X_\xi(L,T_{max})=1$, which means $X(\xi,T_{max})$ is strictly monotonic in a small neighborhood of $L$. Hence,
$$X(L,T_{max})>X(\xi,T_{max})~\textrm{ for }~\xi\in[-L,L).$$
From this we know, if $x\geq X(L,T_{max})$, we have $x-X(\xi,T_{max})>0$ for $\xi\in(-L,L)$.
From Theorem \ref{uniform bv}, we know
$$u(x,T_{max})=\int_{-L}^LG(x-X(\theta,T_{max}))m_0(\theta)d\theta.$$
Thus
\begin{align*}
u_x(x,T_{max})&=\int_{-L}^LG'(x-X(\theta,T_{max}))m_0(\theta)d\theta\\
&=-\int_{-L}^LG(x-X(\theta,T_{max}))m_0(\theta)d\theta=-u(x,T_{max}).
\end{align*}
This shows
$$u(x,T_{max})=u(X(L,T_{max}),T_{max})e^{-x+X(L,T_{max})}.$$
Hence, $u(\cdot ,T_{max})\in C^\infty[X(L,T_{max}),+\infty)$.

Similarly, we can show $u(\cdot ,T_{max})\in C^\infty(-\infty,X(-L,T_{max})]$.

\emph{Step 2.} We only left the case for $x\in O_{T_{max}}$.

When $x\in O_{T_{max}}$, there exists a $\eta\in O$ such that $X(\eta,T_{max})=x$. Because $X_\xi(\eta,T_{max})>0,$ we know $\eta$ is the unique point  satisfying $X(\eta,T_{max})=x.$
Rewrite $u(x,T_{max})$ as
\begin{align*}
u(x,T_{max})=&\int_{\eta}^L G(X(\eta,T_{max})-X(\theta,T_{max}))m_0(\theta)d\theta\\
&+\int_{-L}^\eta G(X(\eta,T_{max})-X(\theta,T_{max}))m_0(\theta)d\theta.
\end{align*}
Using $X_\eta(\eta,T_{max})>0$, we can obtain
\begin{align*}
&\quad u_x(x,T_{max})=\frac{1}{X_\eta(\eta,T_{max})}u_\eta(X(\eta,T_{max}),T_{max})\\
&=\frac{1}{X_\eta(\eta,T_{max})}\bigg(\int_{\eta}^L G'(X(\eta,T_{max})-X(\theta,T_{max}))X_\eta(\eta,T_{max})m_0(\theta)d\theta\\
&-\frac{1}{2}m_0(\eta)+\frac{1}{2}m_0(\eta)+\int^{\eta}_{-L} G'(X(\eta,T_{max})-X(\theta,T_{max}))X_\eta(\eta,T_{max})m_0(\theta)d\theta\bigg).
\end{align*}
Hence,
\begin{align}\label{uxinotmax}
u_x(x,T_{max})=&\int_{\eta}^L G(X(\eta,T_{max})-X(\theta,T_{max}))m_0(\theta)d\theta\nonumber\\
&-\int_{-L}^\eta G(X(\eta,T_{max})-X(\theta,T_{max}))m_0(\theta)d\theta.
\end{align}
Taking derivative again shows that
\begin{align}\label{mtmax}
u_{xx}(x,T_{max})=&-\frac{m_0(\eta)}{2X_\eta(\eta,T_{max})}+\int_{\eta}^L G(X(\eta,T_{max})-X(\theta,T_{max}))m_0(\theta)d\theta\nonumber\\
&-\frac{m_0(\eta)}{2X_\eta(\eta,T_{max})}+\int_{-L}^\eta G(X(\eta,T_{max})-X(\theta,T_{max}))m_0(\theta)d\theta\nonumber\\
=&-\frac{m_0(\eta)}{X_\eta(\eta,T_{max})}+u(x,T_{max}).
\end{align}
Because $m_0\in C_c^1(\mathbb{R})$ and $X_\xi(\cdot ,T_{max})\in C^1(-L,L)$, which implies
$$u(\cdot ,T_{max})\in C^{3}(O_{T_{max}}).$$
Together with Step 1 and Step 2, we obtain
$$u(\cdot ,T_{max})\in C^{3}(\mathbb{R}\setminus F_{T_{max}}).$$

\emph{Step 3.} Because $\mathbb{R}\setminus F_{T_{max}}$ is an open set, for any $\phi\in C_c^\infty(\mathbb{R}\setminus F_{T_{max}})$ we have
\begin{align*}
\int_{\mathbb{R}}\phi(x)m(dx,T_{max})&=\int_{\mathbb{R}}u(x,T_{max})\phi(x)+u_x(x,T_{max})\phi_x(x)dx\\
&=\int_{\mathbb{R}\setminus F_{T_{max}}}u(x,T_{max})\phi(x)+u_x(x,T_{max})\phi_x(x)dx\\
&=\int_{\mathbb{R}\setminus F_{T_{max}}}(u(x,T_{max})-u_{xx}(x,T_{max}))\phi(x)dx,
\end{align*}
where \eqref{radon} was used. Because $\phi$ is arbitrary and $u(\cdot ,T_{max})\in C^{3}(\mathbb{R}\setminus F_{T_{max}})$, we obtain
\begin{align}\label{mrepresentation}
m(\cdot ,T_{max})=u(\cdot ,T_{max})-u_{xx}(\cdot ,T_{max})\in C^{1}(\mathbb{R}\setminus F_{T_{max}}).
\end{align}
From Theorem \ref{RadonmeasureatT}, we know $m(\cdot,T_{max})$ has compact support in $(-L,L)$. Hence,
$$m(\cdot ,T_{max})\in C^{1}(O_{T_{max}}).$$
Because the Radon measure  $m(\cdot ,T_{max})$ has finite total variation, we obtain
$$m(\cdot ,T_{max})\in L^1(O_{T_{max}}).$$

From \eqref{mtmax}, we know
$$m(x,T_{max})=\frac{m_0(\eta)}{X_\eta(\eta,T_{max})}$$
where $x\in O_{T_{max}}$ and $X(\eta,T_{max})=x.$ This means \eqref{keepsign} holds in the set $O$:
$$m(X(\xi,T_{max}),T_{max})X_\xi(\xi,T_{max})=m_0(\xi)~\textrm{ for }~\xi\in O.$$

This finishes our proof.

\end{proof}

Because $u(\cdot ,T_{max})$ and $u_x(\cdot ,T_{max})$ are BV functions, their discontinuous points are countable.
We give a proposition to show  discontinuous points of $u_x(\cdot ,T_{max})$.
First, let us introduce two subsets of $F_{T_{max}}$.
$$\widetilde{F}_{T_{max}}=\{x\in F_{T_{max}}:X^{-1}(x,T_{max})=\{\xi\} ~\textrm{ for some }~\xi\in[-L,L]\},$$
and
\begin{align}\label{FTmax2}
\widehat{F}_{T_{max}}=\{x\in F_{T_{max}}:X^{-1}(x,T_{max})=[\xi_1,\xi_2] ~\textrm{ for some }~\xi_1<\xi_2\}.
\end{align}

\begin{proposition}\label{discontinuousset}
Let the assumptions in Theorem \ref{uniform bv} hold. Then, $u_x(\cdot ,T_{max})\in C(\mathbb{R}\setminus \widehat{F}_{T_{max}})$ and $u_x(\cdot ,T_{max})$ is not continuous at $y\in \widehat{F}_{T_{max}}$.
\end{proposition}
\begin{proof}
\emph{Step 1.} Assume $y\in \widetilde{F}_{T_{max}}$ and we prove $u_x(\cdot ,T_{max})$ is continuous at $y$.

By definition of $F_{T_{max}}$, we know there is only one point $\xi_0\in F$, such that $X(\xi_0,T_{max})=y$. Due to \eqref{eq:dense}, there exist two sequence $\{\overline{y}_n\}$ and $\{\widehat{y}_n\}$ such that the following hold:
$$\{\overline{y}_n\}\subset O_{T_{max}},\quad \lim_{n\rightarrow +\infty}\overline{y}_n=y,\quad \overline{y}_n\textrm{ is increasing}$$
and
$$\{\widehat{y}_n\}\subset O_{T_{max}},\quad \lim_{n\rightarrow +\infty}\widehat{y}_n=y,\quad \widehat{y}_n\textrm{ is decreasing.}$$
Because $\overline{y}_n\in O_{T_{max}}$, there is a unique  $\overline{\xi}_n\in O$ such that $X(\overline{\xi}_n,T_{max})=\overline{y}_n$. Similarly, we have a unique $\widehat{\xi}_n\in O$ such that $X(\widehat{\xi}_n,T_{max})=\widehat{y}_n.$  (Uniqueness is because $X(\xi,T_{max})$ is strictly monotonic in $O$.)   Moreover, we have
$$\overline{\xi}_n<\xi_0<\widehat{\xi}_n,$$
and
$$\lim_{n\rightarrow+\infty}\overline{\xi}_n=\xi_0=\lim_{n\rightarrow +\infty}\widehat{\xi}_n.$$

Because formula \eqref{uxinotmax} holds for $x\in O_{T_{max}}$,  we know
\begin{align*}
u_x(\overline{y}_n,T_{max})=&\int_{\overline{\xi}_n}^L G(X(\overline{\xi}_n,T_{max})-X(\theta,T_{max}))m_0(\theta)d\theta\nonumber\\
&-\int_{-L}^{\overline{\xi}_n} G(X(\overline{\xi}_n,T_{max})-X(\theta,T_{max}))m_0(\theta)d\theta.
\end{align*}
Let $n$ goes to infinity and we obtain
\begin{align*}
u_x(y-,T_{max})=&\int_{\xi_0}^L G(X(\xi_0,T_{max})-X(\theta,T_{max}))m_0(\theta)d\theta\nonumber\\
&-\int_{-L}^{\xi_0} G(X(\xi_0,T_{max})-X(\theta,T_{max}))m_0(\theta)d\theta.
\end{align*}
Similarly, we have
\begin{align*}
u_x(\widehat{y}_n,T_{max})=&\int_{\widehat{\xi}_n}^L G(X(\widehat{\xi}_n,T_{max})-X(\theta,T_{max}))m_0(\theta)d\theta\nonumber\\
&-\int_{-L}^{\widehat{\xi}_n} G(X(\widehat{\xi}_n,T_{max})-X(\theta,T_{max}))m_0(\theta)d\theta.
\end{align*}
and
\begin{align*}
u_x(y+,T_{max})=&\int_{\xi_0}^L G(X(\xi_0,T_{max})-X(\theta,T_{max}))m_0(\theta)d\theta\nonumber\\
&-\int_{-L}^{\xi_0} G(X(\xi_0,T_{max})-X(\theta,T_{max}))m_0(\theta)d\theta.
\end{align*}
This implies $u_x(y-,T_{max})=u_x(y+,T_{max})$.
For any $y\in \widetilde{F}_{T_{max}}$, define
\begin{align*}
u_x(y,T_{max})=&\int_{\xi_0}^L G(X(\xi_0,T_{max})-X(\theta,T_{max}))m_0(\theta)d\theta\nonumber\\
&-\int_{-L}^{\xi_0} G(X(\xi_0,T_{max})-X(\theta,T_{max}))m_0(\theta)d\theta.
\end{align*}
Then using similar argument for any sequence $\mathbb{R}\setminus \widehat{F}_{T_{max}}\ni y_n\rightarrow y$,  we know
$$u_x(\cdot ,T_{max})\in C(\mathbb{R}\setminus \widehat{F}_{T_{max}}).$$

\emph{Step 2.} Assume $y\in \widehat{F}_{T_{max}}$ and we prove $u_x(\cdot ,T_{max})$ is discontinuous at $y$.

Set
$$\xi_1=\min\{\xi\in F:X(\xi,T_{max})=y\}~\textrm{ and }~\xi_2=\max\{\xi\in F:X(\xi,T_{max})=y\}.$$
By definition of $\widehat{F}_{T_{max}}$ we know $\xi<\xi_2$. Moreover, we know
$$X(\xi,T_{max})=y,\quad X_\xi(\xi,T_{max})=0~\textrm{ for }~\xi\in[\xi_1,\xi_2].$$
 \textbf{Claim:} $m_0$ will not change sign in $[\xi_1,\xi_2]$.

If this is not true, then we have $\eta\in [\xi_1,\xi_2]$ such that $m_0(\eta)=0$. Remark \ref{positivenegative} tells us that $X_\xi(\eta,T_{max})=1$ and we obtain a contradiction.

Similar to Step 1, we have four sequences $\overline{\overline{y}}_n$, $\overline{\overline{\xi}}_n$, $\widehat{\widehat{y}}_n$ and $\widehat{\widehat{\xi}}_n$ which satisfy
$$\lim_{n\rightarrow+\infty}\overline{\overline{y}}_n=y=\lim_{n\rightarrow+\infty}\widehat{\widehat{y}}_n,$$
$$\overline{\overline{y}}_n\in O_{T_{max}}\textrm{ increasing},\quad \widehat{\widehat{y}}_n\in O_{T_{max}}\textrm{ decreasing},$$
and
$$\lim_{n\rightarrow+\infty}\overline{\overline{\xi}}_n=\xi_1,\quad \lim_{n\rightarrow+\infty}\widehat{\widehat{\xi}}_n=\xi_2.$$
From \eqref{uxinotmax}, we know
\begin{align*}
u_x(\overline{\overline{y}}_n,T_{max})=&\int_{\overline{\overline{\xi}}_n}^L G(X(\overline{\overline{\xi}}_n,T_{max})-X(\theta,T_{max}))m_0(\theta)d\theta\nonumber\\
&-\int_{-L}^{\overline{\overline{\xi}}_n} G(X(\overline{\overline{\xi}}_n,T_{max})-X(\theta,T_{max}))m_0(\theta)d\theta.
\end{align*}
Let $n$ go to $+\infty$ and we obtain
\begin{align*}
u_x(y-,T_{max})=&\int_{\xi_1}^L G(X(\xi_1,T_{max})-X(\theta,T_{max}))m_0(\theta)d\theta\nonumber\\
&-\int_{-L}^{\xi_1} G(X(\xi_1,T_{max})-X(\theta,T_{max}))m_0(\theta)d\theta\\
=\int_{\xi_1}^L &G(y-X(\theta,T_{max}))m_0(\theta)d\theta-\int_{-L}^{\xi_1} G(y-X(\theta,T_{max}))m_0(\theta)d\theta.
\end{align*}
Similarly, we also have
\begin{align*}
u_x(y+,T_{max})=\int_{\xi_2}^L G(y-X(\theta,T_{max}))m_0(\theta)d\theta-\int_{-L}^{\xi_2} G(y-X(\theta,T_{max}))m_0(\theta)d\theta.
\end{align*}
Hence, using the above claim, we have
\begin{align}\label{peakonweight}
u_x(y-,T_{max})-u_x(y+,T_{max})&=2\int_{\xi_1}^{\xi_2}G(y-X(\theta,T_{max}))m_0(\theta)d\theta\nonumber\\
&=\int_{\xi_1}^{\xi_2}m_0(\theta)d\theta\neq0
\end{align}
which shows that $u_x(\cdot ,T_{max})$ is not continuous at $y$.

\end{proof}

Next, we prove Theorem \ref{maintheorem3}.  Let's give some  notations first.

Assume $F_{T_{max}}=\{x_i\}_{i=1}^{N_1}$  and  $x_1<x_2\ldots<x_{N_1}$. Let  $\widehat{F}_{T_{max}}=\{x_i\}_{i=1}^N$ ($N\leq N_1$).
From the proof \eqref{peakonweight}, we know that for each $1\leq i\leq N$ there exist $\xi_{i1}<\xi_{i2}$ such that
$$u_x(x_i-,T_{max})-u_x(x_i+,T_{max})=p_i$$
where
\begin{align}\label{peakonweights}
p_i=\int_{\xi_{i1}}^{\xi_{i2}}m_0(\theta)d\theta.
\end{align}

Set
\begin{align}\label{definitionofm1}
m_1(x)=\left\{
         \begin{array}{ll}
           m(x,T_{max}),\quad x\in O_{T_{max}}; \\
           0,\quad x\in \mathbb{R}\setminus O_{T_{max}}.
         \end{array}
       \right.
\end{align}

\begin{proof}[Proof of Theorem \ref{maintheorem3}]

For any text function $\phi\in C_c^\infty(\mathbb{R})$, we have
\begin{align*}
\int_{\mathbb{R}}\phi(x)m(dx,T_{max})=\int_{\mathbb{R}}u(x,T_{max})\phi(x)+u_x(x,T_{max})\phi_x(x)dx\\
=\bigg(\int_{-\infty}^{x_1}+\sum_{i=1}^{N_1-1}\int_{x_i}^{x_{i+1}}+\int^{+\infty}_{x_{N_1}}\bigg)\Big[u(x,T_{max})\phi(x)+u_x(x,T_{max})\phi_x(x)\Big]dx.
\end{align*}
Because $u_x(\cdot ,T_{max})\in C^{k+2}(\mathbb{R}\setminus F_{T_{max}})$, integration by parts leads to
\begin{align*}
&\int_{\mathbb{R}}\phi(x)m(dx,T_{max})\\
&=\bigg(\int_{-\infty}^{x_1}+\sum_{i=1}^{N_1-1}\int_{x_i}^{x_{i+1}}+\int^{+\infty}_{x_{N_1}}\bigg)\Big[\big(u(x,T_{max})-u_{xx}(x,T_{max})\big)\phi(x)\Big]dx\\
&\qquad+\sum_{i=1}^{N_1}\big(u_x(x_i-,T_{max})-u_x(x_i+,T_{max})\big)\phi(x_i)\\
&=\int_{O_{T_{max}}}m(x,T_{max})\phi(x)dx+\sum_{i=1}^{N_1}\big(u_x(x_i-,T_{max})-u_x(x_i+,T_{max})\big)\phi(x_i).
\end{align*}
Because $u_x(\cdot ,T_{max})$ is continuous at $x_i$ for $i\geq N+1$, combining \eqref{mrepresentation} and \eqref{peakonweights} gives that
\begin{align*}
\int_{\mathbb{R}}\phi(x)m(dx,T_{max})&=\int_{O_{T_{max}}}m(x,T_{max})\phi(x)dx+\sum_{i=1}^{N}\int_{\xi_{i1}}^{\xi_{i2}}m_0(\theta)d\theta\phi(x_i)\\
&=\int_{O_{T_{max}}}m(x,T_{max})\phi(x)dx+\sum_{i=1}^{N}p_i\phi(x_i)\\
&=\int_{O_{T_{max}}}m(x,T_{max})\phi(x)dx+\sum_{i=1}^{N}\int_{\mathbb{R}}p_i\delta(x-x_i)\phi(x)dx\\
&=\int_{\mathbb{R}}\bigg(m_1(x)+\sum_{i=1}^Np_i\delta(x-x_i)\bigg)\phi(x)dx.
\end{align*}

\end{proof}
This theorem tells us that peakons are exactly the points in the set $\widehat{F}_{T_{max}}$. Hence, a peakon is formulated when some Lagrangian labels in a interval $[\xi_1,\xi_2]$ aggregate into one point at $T_{max}$ and the weight of the peakon is the integration of $m_0(x)$ on $[\xi_1,\xi_2]$.

\section{Solutions after blow-up.}\label{sec5}
At the blow up time, the solution to the mCH equation $m$ becomes a Radon measure. In this section, we assume initial data $m_0$ belongs to the Radon measure space $\mathcal{M}(\mathbb{R})$ and use the Lagrange dynamics to prove that weak solution to \eqref{mCH}-\eqref{initial m} exists  globally in Radon measure space.

\subsection{Regularized Lagrange dynamics and BV estimate.}
Let $m_0\in\mathcal{M}(\mathbb{R})$ satisfies
\begin{align}\label{initialRadon}
\mathrm{supp}\{m_0\}\subset(-L,L)~\textrm{ and }~M_1:=|m_0|(\mathbb{R})<+\infty.
\end{align}
$G'$ is not continuous and may not be integrable with respect to Radon measure $m_0$. \eqref{Lagran dynamics} can not be used directly.  Hence, a regularization is needed.

Let's give the definition of mollifier.

\begin{definition}\label{defmollifier}
$\mathrm{(i)}$ Define the mollifier $0\leq\rho\in C^k(\mathbb{R})$, $k\geq2$ satisfying
\begin{align*}
\int_{\mathbb{R}}\rho(x)dx=1,\quad \rho(x)=\rho(|x|) ~\textrm{ and }~\mathrm{supp}\{\rho\}\subset\{x\in\mathbb{R}:|x|<1\}.
\end{align*}
$\mathrm{(ii)}$ For each $\epsilon>0$, set
$$\rho_\epsilon(x):=\frac{1}{\epsilon}\rho(\frac{x}{\epsilon}).$$
\end{definition}
With this definition, we define
$$G^\epsilon (x):=(\rho_\epsilon\ast G)(x).$$
Hence, $G^\epsilon\in C^k(\mathbb{R})$ for $k\geq2$. By Young's inequality we have
\begin{align}\label{Gproperties1}
||G^\epsilon||_{L^\infty}\leq ||G||_{L^\infty}=\frac{1}{2}, \quad ||G_x^\epsilon||_{L^\infty}\leq ||G_x||_{L^\infty}=\frac{1}{2}
\end{align}
and
\begin{align*}
||G^\epsilon||_{L^1}\leq ||G||_{L^1}=1, \quad ||G_x^\epsilon||_{L^1}\leq ||G_x||_{L^\infty}=1.
\end{align*}

Because $G_{xx}(x)=G(x)$ when $x\neq0$, we have
$$|G^\epsilon_{xx}(x)|=\bigg|\int_{\mathbb{R}}\rho_\epsilon(y)G_{xx}(x-y)dy\bigg|=\bigg|\int_{\mathbb{R}}\rho_\epsilon(y)G(x-y)dy\bigg|\leq \frac{1}{2},~\textrm{ for }~|x|>\epsilon.$$
 On the other hand, because $G^\epsilon_{xx}\in C[-\epsilon,\epsilon]$, there is a constant $\ell^\epsilon>0$ such that
$$G^\epsilon_{xx}(x)\leq \ell^\epsilon~\textrm{ for }~x\in [-\epsilon,\epsilon].$$
Hence, $G^\epsilon_x(x)$ is a global Lipschitz function.
For any measurable function $X(\xi,t)$, we define
\begin{align*}
U_\epsilon(x;X):=\bigg(\int_{-L}^LG^\epsilon(x-X(\theta,t))dm_0(\theta)\bigg)^2-\bigg(\int_{-L}^LG_x^\epsilon(x-X(\theta,t))dm_0(\theta)\bigg)^2
\end{align*}
and
\begin{align*}
U^\epsilon(x;X):=[\rho_\epsilon\ast U_\epsilon](x;X).
\end{align*}
The regularized Lagrange dynamics is given by
\begin{align*}
\left\{
  \begin{array}{ll}
    \dot{X}(\xi,t)=U^\epsilon(X(\xi,t);X), \\
    \displaystyle{X(\xi,0)=\xi\in [-L,L].}
  \end{array}
\right.
\end{align*}
Consider this equation in the Banach space $C[-L,L]$ with $\sup$ norm. One can easily show that the vector field is globally Lipschitz. Hence, by the Picard theorem for ODEs in a Banach space, we obtain a unique global solution
$$X^\epsilon(\xi,t)\in C([-L,L]\times[0,+\infty))~\textrm{ for any }~\epsilon>0.$$

Define
\begin{align}\label{umepsilon}
u^\epsilon(x,t):=\int_{-L}^LG^\epsilon(x-X^\epsilon(\theta,t))dm_0(\theta),~~m^\epsilon(x,t):=u^\epsilon(x,t)-u^\epsilon_{xx}(x,t)
\end{align}
and
\begin{align}\label{mepsilon}
m_\epsilon(\cdot,t):=X^\epsilon(\cdot,t)\# m_0(\cdot).
\end{align}
By the definition, we have
\begin{align*}
u^\epsilon(x,t)=\int_{-L}^LG^\epsilon(x-X^\epsilon(\theta,t))dm_0(\theta)=\int_{\mathbb{R}}G^\epsilon(x-y)m_\epsilon(dy,t).
\end{align*}
Hence, we have the following relation between $m^\epsilon$ and $m_\epsilon$
\begin{align}\label{relationm}
m^\epsilon(x,t)=(1-\partial_{xx})\int_{\mathbb{R}}G^\epsilon(x-y)m_\epsilon(dy,t)=\int_{\mathbb{R}}\rho_\epsilon(x-y)m_\epsilon(dy,t).
\end{align}
In the following of this paper, we  denote
$$U_\epsilon(x,t):=(u^\epsilon)^2(x,t)-(u_x^\epsilon)^2(x,t)~\textrm{ and }~U^\epsilon(x,t):=[\rho_\epsilon\ast U_\epsilon](x,t).$$
Hence, we have
\begin{align}\label{lagrange2}
\dot{X}^\epsilon(\xi,t)=U^\epsilon(X^\epsilon(\xi,t),t).
\end{align}

From Definition \ref{BV} we can easily obtain
\begin{align*}
\mathrm{Tot.Var.}\{G\}=1,\quad \mathrm{Tot.Var.}\{G_x\}=2
\end{align*}
and
\begin{align}\label{Gproperties3}
\mathrm{Tot.Var.}\{G^\epsilon\}=1,\quad \mathrm{Tot.Var.}\{G_x^\epsilon\}=2.
\end{align}

We have the following Lemma about $u^\epsilon$.
\begin{lemma}\label{uepsilonlemma}
Let $m_0\in \mathcal{M}(\mathbb{R})$ satisfy \eqref{initialRadon}. For $\epsilon>0$, $u^\epsilon(x,t)$ is defined by \eqref{umepsilon}. Then, the following statements hold:\\
$\mathrm{(i)}$
$$||u^\epsilon||_{L^\infty}\leq\frac{1}{2}M_1~\textrm{ and }~||u_x^\epsilon||_{L^\infty}\leq\frac{1}{2}M_1~\textrm{ uniformly in }~\epsilon.$$
$\mathrm{(ii)}$
$$\mathrm{Tot.Var.}\{u^\epsilon(\cdot,t)\}\leq M_1~\textrm{ and }~\mathrm{Tot.Var.}\{u^\epsilon(\cdot,t)\}\leq 2M_1~\textrm{ uniformly in }~\epsilon.$$
$\mathrm{(iii)}$
For any $t,~s\in [0,\infty)$, we have
$$\int_\mathbb{R}|u^\epsilon(x,t)-u^\epsilon(x,s)|dx\leq \frac{1}{2}M_1^3|t-s|~\textrm{ and }~\int_\mathbb{R}|u_x^\epsilon(x,t)-u_x^\epsilon(x,s)|dx\leq M_1^3|t-s|.$$

Moreover, for any $T>0$, there exist  subsequences of $u^\epsilon$, $u_x^\epsilon$ (also denoted as $u^\epsilon$, $u_x^\epsilon$) and two functions $u,~u_x\in BV(\mathbb{R}\times[0,T))$ such that
\begin{align*}
u^\epsilon\rightarrow u,~~u_x^\epsilon\rightarrow u_x~\textrm{ in }~L_{loc}^1(\mathbb{R}\times[0,+\infty))~\textrm{ as }~\epsilon\rightarrow0
\end{align*}
and  $u$, $u_x$ satisfy all the properties in  $\mathrm{(i)}$, $\mathrm{(ii)}$ and $\mathrm{(iii)}$.
\end{lemma}
\begin{proof}
(i) From \eqref{Gproperties1} and the definition of $u^\epsilon$, we can easily obtain (i).

(ii)
For any $\{x_i\}\subset\mathbb{R}$, $x_i<x_{i+1}$, \eqref{Gproperties3} yields
\begin{align*}
\sum_i|u^\epsilon(x_i,t)-u^\epsilon(x_{i-1},t)|\leq &\int_{-L}^L \sum_i|G^\epsilon(x_i-X^\epsilon(\theta,t))-G^\epsilon(x_{i-1}-X^\epsilon(\theta,t))|dm_0(\theta)\\
\leq &\mathrm{Tot.Var.}\{G^\epsilon\}M_1=M_1.
\end{align*}
Hence, $\mathrm{Tot.Var.}\{u^\epsilon(\cdot,t)\}\leq M_1$. Similarly, we can obtain $\mathrm{Tot.Var.}\{u_x^\epsilon(\cdot,t)\}\leq 2M_1$.

(iii)
\begin{align*}
\int_\mathbb{R}|u^\epsilon(x,t)-u^\epsilon(x,s)|dx\leq \int_{\mathbb{R}}\int_{-L}^L |G^\epsilon(x-X^\epsilon(\theta,t))-G^\epsilon(x-X^\epsilon(\theta,s))|dm_0(\theta)dx.
\end{align*}
By the definition of $U^\epsilon$ and \eqref{lagrange2}, we know
$$|\dot{X}^\epsilon(\xi,t)|\leq \frac{1}{2}M_1^2.$$
Hence,
$$|X^\epsilon(\theta,t)-X^\epsilon(\theta,s)|\leq\frac{1}{2}M_1^2|t-s|.$$
 \cite[Lemma 2.3]{Bressan} gives
\begin{align*}
\int_{\mathbb{R}}|G^\epsilon(x-X^\epsilon(\theta,t))-G^\epsilon(x-X^\epsilon(\theta,s))|dx\leq \mathrm{Tot.Var.}\{G^\epsilon\}|X^\epsilon(\theta,t)-X^\epsilon(\theta,s)|\leq\frac{1}{2}M_1^2|t-s|.
\end{align*}
Hence
\begin{align*}
\int_\mathbb{R}|u^\epsilon(x,t)-u^\epsilon(x,s)|dx\leq \frac{1}{2}M_1^3|t-s|.
\end{align*}

Similarly, we can obtain
\begin{align*}
\int_\mathbb{R}|u_x^\epsilon(x,t)-u_x^\epsilon(x,s)|dx\leq M_1^3 |t-s|.
\end{align*}

The rest results can be obtained by using \cite[Theorem 2.4,2.6]{Bressan}.

\end{proof}

\subsection{Weak consistency and convergence theorem}
In this subsection, we show that $u^\epsilon$ defined by \eqref{umepsilon} is weak consistent with the mCH equation \eqref{mCH}-\eqref{initial m}.

We rewrite \eqref{mCH} as equation of $u$,
\begin{align*}
&\quad(1-\partial_{xx})u_t+[(u^2-u^2_x)(u-u_{xx})]_x\\
&=(1-\partial_{xx})u_t+(u^3+uu_x^2)_x-\frac{1}{3}(u^3)_{xxx}+\frac{1}{3}(u^3_x)_{xx}=0.
\end{align*}

Now, we introduce the definition of weak solution in terms of $u$. To this end, for  $\phi\in C_c^\infty(\mathbb{R}\times[0,T))$, we denote the functional
\begin{align}\label{weakfunctional}
\mathcal{L}(u,\phi):&=\int_0^T\int_{\mathbb{R}}u(x,t)[\phi_t(x,t)-\phi_{txx}(x,t)]dxdt\nonumber\\
&\quad-\frac{1}{3}\int_0^T\int_{\mathbb{R}}u^3_x(x,t)\phi_{xx}(x,t)dxdt-\frac{1}{3}\int_0^T\int_{\mathbb{R}}u^3(x,t)\phi_{xxx}(x,t)dxdt\nonumber\\
&\quad+\int_0^T\int_{\mathbb{R}}(u^3+uu_x^2)\phi_x(x,t)dxdt.
\end{align}
Then, the definition of the weak solution to \eqref{mCH} in terms of $u(x,t)$ is given as follows.
\begin{definition}\label{weaksolution}
For $m_0\in\mathcal{M}(\mathbb{R})$, a function
$$u\in C([0,T);H^1(\mathbb{R}))\cap L^\infty(0,T; W^{1,\infty}(\mathbb{R}))$$
 is said to be a weak solution of \eqref{mCH}-\eqref{initial m} if
$$\mathcal{L}(u,\phi)=-\int_{\mathbb{R}}\phi(x,0)dm_0(x)$$
holds for all $\phi\in C_c^\infty(\mathbb{R}\times[0,T))$.  If $T=+\infty$, we call $u(x,t)$ as a global weak solution of the mCH equation.
\end{definition}

For simplicity in notations, we denote
$$\langle f(x,t), g(x,t)\rangle:=\int_0^T\int_{\mathbb{R}}f(x,t)g(x,t)dxdt.$$
For any test function $\phi\in C_c^\infty(\mathbb{R}\times[0,T))$, we have
\begin{align}\label{consistence}
&\quad\langle m_\epsilon(x,t),\phi_t\rangle+\langle U^\epsilon m_\epsilon,\phi_x\rangle\nonumber\\
&=\int_0^T\int_{\mathbb{R}}\phi_t(x,t)m_\epsilon(dx,t)dt+\int_0^T\int_{\mathbb{R}}U^\epsilon(x,t)\phi_x(x,t)m_\epsilon(dx,t)dt\nonumber\\
&=\int_0^T\int_{-L}^L\Big[\phi_t(X^\epsilon(\theta,t),t)+U^\epsilon(X^\epsilon(\theta,t),t)\phi_x(X^\epsilon(\theta,t),t)\Big]dm_0 (\theta) dt\nonumber\\
&=\int_0^T\frac{d}{dt}\int_{-L}^L\phi(X^\epsilon(\theta,t),t)dm_0(\theta)  dt=-\int_{-L}^L\phi(\theta,0)dm_0(\theta).
\end{align}
On the other hand, combining  the definition  \eqref{umepsilon}  and \eqref{weakfunctional} gives
\begin{align*}
\mathcal{L}(u^\epsilon,\phi)&=\int_0^T\int_{\mathbb{R}}u^\epsilon[\phi_t-\phi_{txx}]dxdt-\frac{1}{3}\int_0^T\int_{\mathbb{R}}(\partial_xu^\epsilon)^3\phi_{xx}dxdt\\
&\qquad\qquad-\frac{1}{3}\int_0^T\int_{\mathbb{R}}(u^\epsilon)^3\phi_{xxx}dxdt+\int_0^T\int_{\mathbb{R}}((u^\epsilon)^3+u^\epsilon(u_x^\epsilon)^2)\phi_xdxdt\\
&=\langle\phi_t,(1-\partial_{xx})u^\epsilon\rangle+\langle[(u^\epsilon)^2-(\partial_xu^\epsilon)^2](1-\partial_{xx})u^\epsilon,\phi_x\rangle\\
&=\langle m^\epsilon,\phi_t\rangle+\langle  U_\epsilon m^\epsilon, \phi_x\rangle.
\end{align*}
Combining the last two equalities, we define
\begin{align}\label{Eepsilon}
E_\epsilon:=\langle m^\epsilon-m_\epsilon,\phi_t\rangle+\langle  U_\epsilon m^\epsilon-U^\epsilon m_\epsilon, \phi_x\rangle=\mathcal{L}(u^\epsilon,\phi)+\int_{\mathbb{R}}\phi(x,0)dm_0(x).
\end{align}
We now state the main result of this section.
\begin{proposition}\label{consistency}
We have the following estimate
$$|E_\epsilon|\leq C\epsilon.$$
The constant $C$ is independent of $\epsilon.$
\end{proposition}
\begin{proof}
By the definition of $m^\epsilon$ and $m_\epsilon$, the first term in \eqref{Eepsilon} can be estimated as
\begin{align*}
\langle m^\epsilon-m_\epsilon,\phi_t\rangle&=\int_0^T\bigg(\int_{\mathbb{R}}\phi_t(x,t)m^\epsilon(x,t)dx-\int_{\mathbb{R}}\phi_t(x,t)m_\epsilon(dx,t)\bigg)dt\\
&=\int_0^T\bigg(\int_{\mathbb{R}}\int_{\mathbb{R}}\phi_t(x,t)\rho_\epsilon(x-y)m_\epsilon(dy,t)dx-\int_{\mathbb{R}}\phi_t(y,t)m_\epsilon(dy,t)\bigg)dt\\
&=\int_0^T\bigg(\int_{\mathbb{R}}\int_{\mathbb{R}}\big[\phi_t(x,t)-\phi_t(y,t)\big]\rho_\epsilon(x-y)m_\epsilon(dy,t)dx\bigg)dt\\
&=\int_0^T\bigg(\int_{\mathbb{R}}\int_{-L}^L\big[\phi_t(x,t)-\phi_t(X^\epsilon(\theta,t),t)\big]\rho_\epsilon(x-X^\epsilon(\theta,t))dm_0(\theta)dx\bigg)dt\\
&\leq M_1||\phi_{tx}||_{L^\infty}T\epsilon.
\end{align*}
For the second term of \eqref{Eepsilon}, because $\rho_\epsilon$ is an even function, by the definition of $U^\epsilon$ we can obtain
\begin{align*}
&\quad\langle  U_\epsilon m^\epsilon-U^\epsilon m_\epsilon, \phi_x\rangle\\
&=\int_0^T\int_{\mathbb{R}}\int_{-L}^LU_\epsilon(x,t)\phi_x(x,t)\rho_\epsilon(x-X^\epsilon(\theta,t))dm_0(\theta)dxdt\\
&\qquad\qquad\qquad-\int_0^T\int_{-L}^LU^\epsilon(X^\epsilon(\theta,t),t)\phi_x(X^\epsilon(\theta,t),t)dm_0(\theta)dt\\
&=\int_0^T\int_{\mathbb{R}}\int_{-L}^LU_\epsilon(x,t)\phi_x(x,t)\rho_\epsilon(x-X^\epsilon(\theta,t))dm_0(\theta)dxdt\\
&\qquad\qquad\qquad-\int_0^T\int_{\mathbb{R}}\int_{-L}^LU_\epsilon(x,t)\rho_\epsilon(x-X^\epsilon(\theta,t))\phi_x(X^\epsilon(\theta,t),t)dm_0(\theta)dxdt\\
&\leq M_1||U_\epsilon||_{L^\infty}||\phi_{xx}||_{L^\infty}T\epsilon\leq \frac{1}{2}M_1^3||\phi_{xx}||_{L^\infty}T\epsilon.
\end{align*}
 This ends the proof.

\end{proof}

Next, we state our main theorem in this section,  which contains Theorem \ref{maintheorem4}.
\begin{theorem}\label{globalweaksolution}
Assume that initial data $m_0\in\mathcal{M}(\mathbb{R})$ satisfies \eqref{initialRadon}. $u^\epsilon(x,t)$ and $m^\epsilon(x,t)$ are defined by \eqref{umepsilon}. Then, the limit function $u$ given by Lemma \ref{uepsilonlemma} is a global weak solution of the mCH equation \eqref{mCH}-\eqref{initial m} and
$$u\in C([0,+\infty);H^1(\mathbb{R}))\cap L^\infty(0,+\infty; W^{1,\infty}(\mathbb{R})).$$
Furthermore, for any $T>0$, we have
$$u\in BV(\mathbb{R}\times[0,T));\quad u_x\in BV(\mathbb{R}\times[0,T)),$$
$$m:=(1-\partial_{xx})u\in \mathcal{M}(\mathbb{R}\times[0,T)),$$
and there exists subsequence of $m^\epsilon$ (also labelled as $m^\epsilon$) such that
$$m^\epsilon\stackrel{\ast}{\rightharpoonup} m\textrm{ in } \mathcal{M}(\mathbb{R}\times[0,T))~\textrm{ as }\epsilon\rightarrow0.$$

\end{theorem}

\begin{proof}
Step 1. Global weak solution.

As it is shown in Lemma \ref{uepsilonlemma}, we have $u,~u_x\in BV(\mathbb{R}\times[0,T))$ such that
\begin{align*}
u^\epsilon\rightarrow u, \quad \partial_xu^\epsilon\rightarrow u_x\textrm{ in } L_{loc}^1(\mathbb{R}\times[0,+\infty)).
\end{align*}
Moreover, for any $T>0$, the limit functions $u,u_x$ satisfy
$$u\in BV(\mathbb{R}\times[0,T)),\quad u_x\in BV(\mathbb{R}\times[0,T)),$$
$$|u(x,t)|\leq\frac{1}{2}M_1,\quad |u_x(x,t)|\leq\frac{1}{2}M_1$$
and
\begin{align*}
\int_{\mathbb{R}}|u(x,t)-u(x,s)|dx\leq \frac{1}{2}M_1^3|t-s|,~\int_{\mathbb{R}}|u_x(x,t)-u_x(x,s)|dx\leq M_1^3|t-s|
\end{align*}
for $t,s\in[0,+\infty)$. Hence,
\begin{align*}
||u(\cdot ,t)-u(\cdot ,s)||_{L^2}^2&=\int_{\mathbb{R}}|u(x,t)-u(x,s)|^2dx\\
&\leq M_1\int_{\mathbb{R}}|u(x,t)-u(x,s)|dx\leq\frac{1}{2}M_1^4|t-s|.
\end{align*}
Similarly, we have
$$||u_x(\cdot ,t)-u_x(\cdot ,s)||_{L^2}^2\leq M_1^4|t-s|.$$
These two inequalities imply
\begin{align*}
||u(\cdot ,t)-u(\cdot ,s)||_{H^1}^2&\leq 2\Big(||u(\cdot ,t)-u(\cdot ,s)||_{L^2}^2+||u_x(\cdot ,t)-u_x(\cdot ,s)||_{L^2}^2\Big)\\
&\leq 3M_1^4|t-s|.
\end{align*}
Therefore
$$u\in C([0,+\infty);H^1(\mathbb{R}))\cap L^\infty(0,+\infty; W^{1,\infty}(\mathbb{R})).$$

For each $\phi\in C_c^\infty(\mathbb{R}\times[0,+\infty))$, there exists $T=T(\phi)$ such that $\phi\in C_c^\infty(\mathbb{R}\times[0,T))$
We now consider convergence for each term of $\mathcal{L}(u^\epsilon,\phi)$,
\begin{align*}
\mathcal{L}(u^\epsilon,\phi)&=\int_0^T\int_{\mathbb{R}}u^\epsilon[\phi_t-\phi_{txx}]dxdt-\frac{1}{3}\int_0^T\int_{\mathbb{R}}(\partial_xu^\epsilon)^3\phi_{xx}dxdt\\
&\quad-\frac{1}{3}\int_0^T\int_{\mathbb{R}}(u^\epsilon)^3\phi_{xxx}dxdt+\int_0^T\int_{\mathbb{R}}((u^\epsilon)^3+u^\epsilon(\partial_xu^\epsilon)^2)\phi_xdxdt.
\end{align*}

For the first term, because supp$\{\phi\}$ is compact,  we can see
\begin{align*}
\int_0^T\int_{\mathbb{R}}u^\epsilon[\phi_t-\phi_{txx}]dxdt\rightarrow \int_0^T\int_{\mathbb{R}}u[\phi_t-\phi_{txx}]dxdt\quad(\epsilon\rightarrow0).
\end{align*}
The second term can be estimated as follows
\begin{align*}
&\quad\int_0^T\int_{\mathbb{R}}[(\partial_xu^\epsilon)^3-u_x^3]\phi_{xx}dxdt\\
&=\int_0^T\int_{\mathbb{R}}(\partial_xu^\epsilon-u_x)[(\partial_xu^\epsilon)^2+u_x^2+\partial_xu^\epsilon u_x]\phi_{xx}dxdt\\
&\leq \frac{3}{4}M_1^2||\phi_{xx}||_{L^\infty}\int_{\mathrm{supp}\{\phi\}}|\partial_xu^\epsilon-u_x|dxdt\rightarrow0\quad (\epsilon\rightarrow0).
\end{align*}
Similarly, we obtain
\begin{align*}
\int_0^T\int_{\mathbb{R}}[(u^\epsilon)^3-u^3]\phi_{xxx}dxdt\rightarrow0\quad(\epsilon\rightarrow0),\\
\int_0^T\int_{\mathbb{R}}[(u^\epsilon)^3-u^3]\phi_{x}dxdt\rightarrow0\quad(\epsilon\rightarrow0),
\end{align*}
and
\begin{align*}
&\quad\int_0^T\int_{\mathbb{R}}[u^\epsilon(\partial_xu^\epsilon)^2-uu_x^2]\phi_xdxdt\\
&=\int_0^T\int_{\mathbb{R}}[(u^\epsilon-u)(\partial_xu^\epsilon)^2+u((\partial_xu^\epsilon)^2-u_x^2)]\phi_xdxdt\\
&=\int_0^T\int_{\mathbb{R}}[(u^\epsilon-u)(\partial_xu^\epsilon)^2+u(\partial_xu^\epsilon+u_x)(\partial_xu^\epsilon-u_x)]\phi_xdxdt\\
&\rightarrow0\quad(\epsilon\rightarrow0).
\end{align*}
Combining the above estimates and Proposition \ref{consistency} gives
$$\mathcal{L}(u,\phi)=-\int_{\mathbb{R}}\phi(x,0)dm_0(x).$$
This proves that $u$ is a global weak solution to the mCH equation.

Step 2.  Now we prove that
$$m^\epsilon\stackrel{\ast}{\rightharpoonup} m\textrm{ in } \mathcal{M}(\mathbb{R}\times[0,T))\quad (\epsilon\rightarrow0).$$

For any test function $\phi\in C_c^1(\mathbb{R}\times[0,T))$, integrating by parts and using the relationship $m^\epsilon=(1-\partial_{xx})u^\epsilon$ imply that
\begin{align*}
\int_0^T\int_{\mathbb{R}}\phi(x,t)dm^\epsilon(x,t)=\int_0^T\int_{\mathbb{R}}\phi(x,t)(1-\partial_{xx})u^\epsilon(x,t)dxdt\\
=\int_0^T\int_{\mathbb{R}}\phi(x,t)u^\epsilon(x,t)+\phi_x(x,t)\partial_xu^\epsilon(x,t)dxdt.
\end{align*}
Taking $\epsilon\rightarrow0$, the right hand side of the above equality converges to
$$\int_0^T\int_{\mathbb{R}}\phi(x,t)u(x,t)+\phi_x(x,t)u_x(x,t)dxdt=\int_0^T\int_{\mathbb{R}}\phi(x,t)m(dx,dt).$$
Hence, $m^{\epsilon}\stackrel{\ast}{\rightharpoonup} m\textrm{ in } \mathcal{M}(\mathbb{R}\times[0,T))$. This ends the proof.

\end{proof}

\begin{remark}
In \cite{GaoLiu}, the authors also provee the total variation stability of $m(\cdot,t)$. That is 
$$|m(\cdot,t)|(\mathbb{R})\leq |m_0|(\mathbb{R}).$$
The weak solution is unique when $u\in L^\infty(0,\infty;W^{2,1}(\mathbb{R})).$ Moreover, examples about nonuniqueness of peakon weak solutions can also be found in \cite{GaoLiu}.
Notice that peakon solutions are not in the solution class $W^{2,1}(\mathbb{R})$.
\end{remark}

\newpage

\bibliographystyle{plain}
\bibliography{bibofLagrangemCH}

\end{document}